\documentclass[leqno]{amsart}
\usepackage[left=1in, right=1in, top=1in, bottom=1in]{geometry}
\usepackage{boilerplate}
\usepackage{mathrsfs}
\usepackage{tikz-cd}
\usepackage{mathdots}
\usetikzlibrary{arrows,arrows.meta}
\usepackage[mathscr]{euscript}
\usepackage{algorithm}
\usepackage{algpseudocode}
\usepackage{hyperref}
\hypersetup{
    colorlinks,
    citecolor=black,
    filecolor=black,
    linkcolor=black,
    urlcolor=black
}
\usepackage{enumitem}
\usepackage{subcaption}

\usepackage{xurl}

\makeatletter
\renewcommand\subsection{\leftskip 0pt\@startsection{subsection}{2}{\z@}%
                                     {-3.25ex\@plus -1ex \@minus -.2ex}%
                                     {1.5ex \@plus .2ex}%
                                     {\normalfont\normalsize\bfseries}}
\renewcommand\subsubsection{\@startsection{subsubsection}{3}{\z@}%
                                     {-3.25ex\@plus -1ex \@minus -.2ex}%
                                     {1.5ex \@plus .2ex}%
                                     {\normalfont\normalsize\bfseries\leftskip 3ex}}
\makeatother

\title{Algorithmic canonical stratifications of simplicial complexes}

\author{Ryo Asai\textsuperscript{\dag}}
\address{Colfax International, 2805 Bowers Ave, Santa Clara, CA 95051, USA}
\email{ryo@colfax-intl.com}
\thanks{\textsuperscript{\dag} Colfax Research}

\author{Jay Shah\textsuperscript{\ddag}}
\address{Fachbereich Mathematik und Informatik, WWU Münster, 48149 M\"{u}nster, Germany}
\email{jayhshah@gmail.com}
\thanks{\textsuperscript{\ddag} Department of Mathematics and Computer Science, University of Münster}

\begin{document}

\tikzcdset{arrow style=tikz, diagrams={>=stealth}}

\begin{abstract} We introduce a new algorithm for the structural analysis of finite abstract simplicial complexes based on local homology. Through an iterative and top-down procedure, our algorithm computes a stratification $\pi$ of the poset $P$ of simplices of a simplicial complex $K$, such that for each strata $P_{\pi=i} \subset P$, $P_{\pi=i}$ is maximal among all open subposets $U \subset \overline{P_{\pi=i}}$ in its closure such that the restriction of the local $\ZZ$-homology sheaf of $\overline{P_{\pi=i}}$ to $U$ is locally constant. Passage to the localization of $P$ dictated by $\pi$ then attaches a canonical stratified homotopy type to $K$.

Using $\infty$-categorical methods, we first prove that the proposed algorithm correctly computes the canonical stratification of a simplicial complex; along the way, we prove a few general results about sheaves on posets and the homotopy types of links that may be of independent interest. We then present a pseudocode implementation of the algorithm, with special focus given to the case of dimension $\leq 3$, and show that it runs in polynomial time. In particular, an $n$-dimensional simplicial complex with size $s$ and $n\leq3$ can be processed in O($s^2$) time or O($s$) given one further assumption on the structure. Processing Delaunay triangulations of $2$-spheres and $3$-balls provides experimental confirmation of this linear running time.
\end{abstract}

\date{\today}
\maketitle

\tableofcontents

\section{Introduction}
Our principal aim in this paper is to detail an algorithm for computing the \emph{canonical stratification} of a simplicial complex in the sense of Nanda \cite{NandaStrat}. Let us first explain the topological significance of this invariant. Suppose that $K$ is a simplicial complex, which we think of as a combinatorial presentation of a topological space. Then one has a variety of combinatorial and algebraic invariants of $K$, among the most basic of which is the \emph{homology} of $K$. The classical Poincar\'{e} duality theorem highlights the centrality and utility of homology as a tool for studying closed manifolds. However, when used to study simplicial complexes that are not closed manifolds, homology often proves to be too coarse of an invariant; to take a simple example, the homology of an $n$-dimensional disk equals that of a point, so homology loses information about dimension. More generally, any invariant of $K$ that is purely an invariant of its homotopy type $\sK$ suffers from this deficiency.

For a finer invariant of $K$, we can instead consider the \emph{local homology} $H_{\ast}(\sigma)$ of a simplex $\sigma \in K$, which is the reduced homology of the compactification of a small open neighborhood about an interior point in $\sigma$. This invariant is sensitive to the local structure of $K$ about $\sigma$, and can, for example, distinguish between spaces of differing dimension, as well as detect the presence of singularities. Local homology therefore presents itself a candidate for addressing some of the deficiencies of homology, and indeed the homotopy type $\sK$ itself. On the other hand, local homology is a \emph{local} invariant recorded for individual simplices, which is undesirable from the perspective of gleaning insight into the entire structure of $K$ as compared to a single, \emph{global} invariant such as $\sK$. We are led to ask:
\\

\textbf{Question}: Can one leverage the information supplied by the assemblage of local homology groups to construct a global invariant of $K$ that refines its homotopy type $\sK$?
\\

For example, let us consider this question in the case where $K$ is the triangulation of a compact $n$-dimensional manifold $M$ with boundary $\partial M$, such as the $2$-disk $D^2$

{\centering
\includegraphics[scale=0.4]{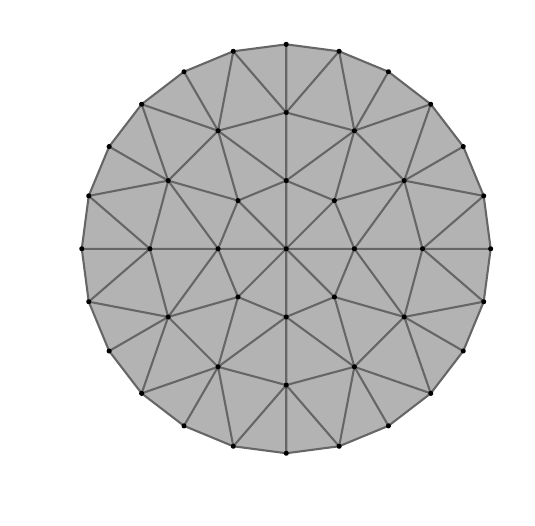}
\par}
Then given a simplex $\sigma \in K$, there are two possibilities for the value of $H_\ast(\sigma)$:

\begin{enumerate}
	\item $H_i(\sigma) = \{ \ZZ \text{ for } i=n \text{ and } $0$ \text{ otherwise} \}$  if and only if points in the interior of $\sigma$ lie in the interior $M - \partial M$. Call $\sigma$ an \emph{interior} simplex.
	\item $H_i(\sigma) = 0$ for all $i$ if and only if all points in $\sigma$ lie on the boundary $\partial M$. Call $\sigma$ a \emph{boundary} simplex.
\end{enumerate}

Now let $P$ be the \emph{poset of simplices} of $K$, so $P$ has for its objects the simplices $\sigma \in K$, with the partial order defined such that $\sigma \leq \tau$ if and only if $\sigma \subset \tau$. We can use local homology to partition $P$ into two subposets, in the following way:
\begin{itemize}
	\item[($\ast$)] Let $[1]$ denote the totally ordered set $\{0<1\}$. Note that for any inclusion $\sigma \subset \tau$, if $\sigma$ is an interior simplex, then $\tau$ is an interior simplex. Therefore, we may define a map of posets $\pi: P \to [1]$ that sends $\sigma$ to $0$ if it is a boundary simplex and $1$ if it is an interior simplex.
\end{itemize}

Then the classifying space of the fiber $P_{\pi=0} \coloneq \{0\} \times_{[1]} P$ is homotopy equivalent to the boundary $\partial M$, while the classifying space of the other fiber $P_{\pi=1} \coloneq \{1\} \times_{[1]} P$ is homotopy equivalent to the interior $M - \partial M$. For example, in the case of the above triangulation of the $2$-disk, we have that the sets of simplices in $P_{\pi=0}$ and $P_{\pi=1}$ correspond to the two subspaces of $K$

{\centering \includegraphics[scale=.4]{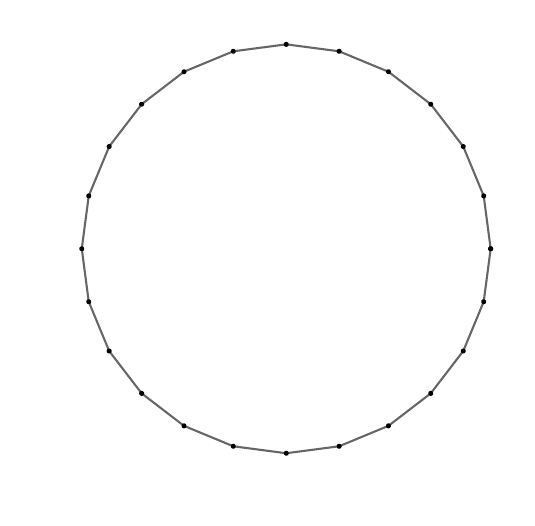} \quad \includegraphics[scale=.4]{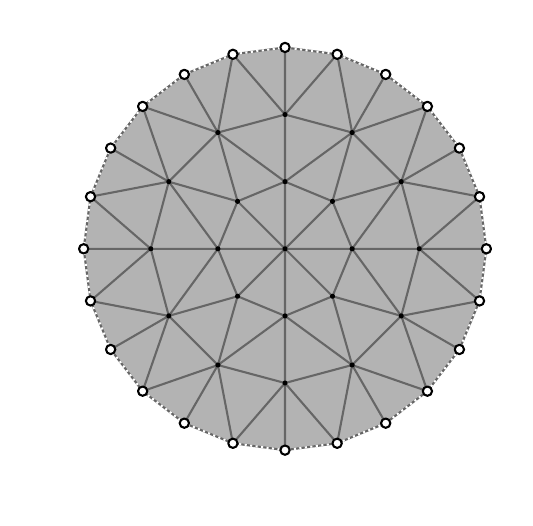} \par}

where the first is $\partial D^2 \simeq S^1$ and the second is its open complement $D^2 - \partial D^2 \simeq \ast$.

Furthermore, the datum of the map $\pi: P \to [1]$ retains more information than just the fibers: we also have the poset of sections $\Fun_{/[1]}([1],P)$, whose classifying space yields the \emph{holink} \cite{Quinn:strat} of the inclusion $\partial M \subset M$, as well as $\sK$ given by the classifying space of $P$ itself. We call the map $\pi$ the \emph{canonical stratification} of $K$ and view it as a discrete presentation of all of this topological information. Moreover, appealing to the higher categorical theory of localization, we can invert the morphisms in the fibers of $\pi$ to produce an $\infty$-category $\sK^{\can}$ equipped with a map $\Pi^{\can}: \sK^{\can} \to [1]$, the \emph{canonical stratified homotopy type} derived from $K$. $\sK^{\can}$ then constitutes our desired refinement of $\sK$.

As a second example, suppose instead that $K$ is a triangulation of the pinched annulus

{\centering
\includegraphics[scale=0.4]{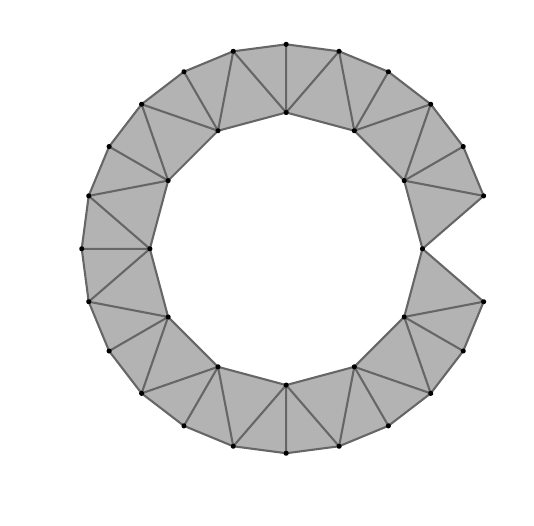}
\par}
\vspace*{-4mm}
Local homology first identifies the interior simplices of $K$ as before, whose removal then yields a subcomplex $L$ which is a triangulation of the wedge of two circles. Next, local homology for the lower $1$-dimensional simplicial complex $L$ identifies those simplices interior in $L$. Finally, removal of these simplices leaves only the vertex of intersection. We can depict this whole process by the sequence of figures

{\centering\includegraphics[scale=.4]{step0.pdf} \includegraphics[scale=.4]{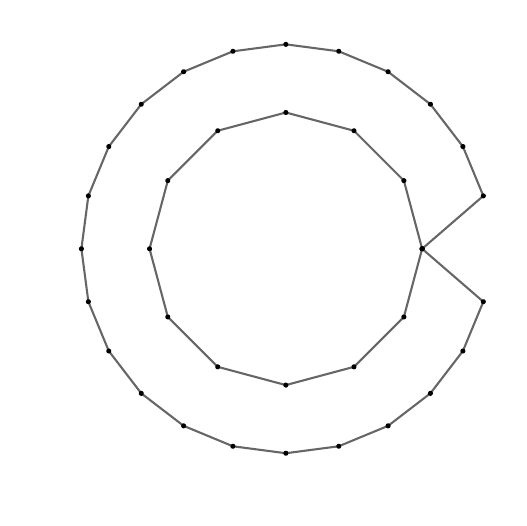}\includegraphics[scale=.4]{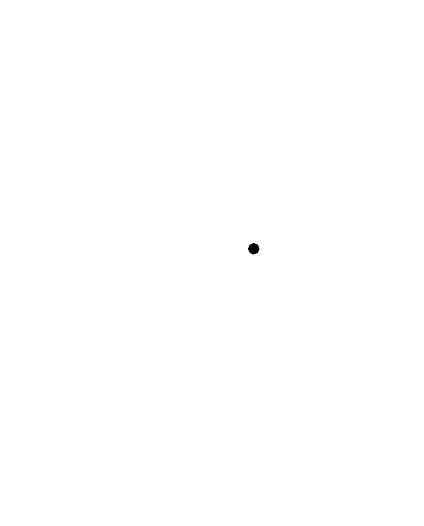}\par}
\vspace*{-4mm}
As before, we can correspondingly construct a map of posets $\pi: P \to [2] \coloneq \{0 < 1 < 2\}$ such that interior simplices in $K$ are sent to $2$, interior simplices in the remainder $L$ are sent to $1$, and the final vertex of intersection is sent to $0$. Collapsing the fibers of $\pi$ to their connected components then yields the poset
\[ \begin{tikzcd}[row sep=3ex, column sep=12ex, text height=.5ex, text depth=0.25ex]
 & \bullet \ar[shift left=1]{rd} & \\
 \bullet \ar[shift left=1]{ru} \ar[shift right=1]{rd} \ar{rr} & & \bullet \\
 & \bullet \ar[shift right=1]{ru}
\end{tikzcd} \]

Furthermore, in this example the $\infty$-category $\sK^{\can}$ obtained by localization is in fact equivalent to the ordinary category
\[ \begin{tikzcd}[row sep=3ex, column sep=12ex, text height=.5ex, text depth=0.25ex]
 & \bullet \ar[shift left=1]{rd}{g_0} & \\
\bullet \ar[shift left=2]{ru}{f_{0}^{+}} \ar{ru}[swap]{f_{0}^{-}} \ar[shift right=2]{rd}[swap, yshift=-1]{f_1^{-}} \ar{rd}[yshift=-1]{f_1^{+}} \ar[shift right=.8]{rr}[swap, yshift=-2.5]{h^{-}} \ar[shift left=.8]{rr}{h^{+}} & & \bullet \\
& \bullet \ar[shift right=1]{ru}[swap]{g_1} &
\end{tikzcd} \]

where the composition is defined by $g_i \circ f^+_i = h^+$ and $g_i \circ f^-_i = h^-$. We can interpret these morphisms as ``exit-paths'', which are homotopy classes of paths where paths and homotopies are constrained to never descend in strata. Zooming in near the vertex of intersection, we may depict these exit-paths as
\begin{center}
\begin{tikzpicture}[scale=.8]
\draw (0,0) -- (-5,2);
\draw (0,0) -- (-5,-2);
\draw (0,0) -- (5,2);
\draw (0,0) -- (5,-2);
\draw[dotted] (-5,2) -- (-5,-2);
\draw[dotted] (5,2) -- (5,-2);
\fill[gray!30] (0,0) -- (-5,2) -- (-5,-2) -- (0,0);
\fill[gray!30] (0,0) -- (5,2) -- (5,-2) -- (0,0);
\draw[black,->,line width=0.5mm] (0,0) -- (-2.5,1) node[black,midway,above] {$f_0^+$};
\draw[black,->,line width=0.5mm] (0,0) -- (2.5,1) node[black,midway,above] {$f_0^-$};
\draw[black,->,line width=0.5mm] (0,0) -- (-2.5,-1) node[black,midway,below] {$f_1^+$};
\draw[black,->,line width=0.5mm] (0,0) -- (2.5,-1) node[black,midway,below] {$f_1^-$};
\draw[black,->,line width=0.5mm] (0,0) -- (-4,0) node[black,midway,above] {$h^+$};
\draw[black,->,line width=0.5mm] (0,0) -- (4,0) node[black,midway,above] {$h^-$};
\draw[black,->,line width=0.5mm] (0,0) -- (4,0) node[black,midway,above] {$h^-$};
\draw[black,->,line width=0.5mm] (3.75,1.5) -- (4,.5) node[black,midway,right] {$g_0$};
\draw[black,->,line width=0.5mm] (3.75,-1.5) -- (4,-.5) node[black,midway,right] {$g_1$};
\draw[black,->,line width=0.5mm] (-3.75,1.5) -- (-4,.5) node[black,midway,left] {$g_0$};
\draw[black,->,line width=0.5mm] (-3.75,-1.5) -- (-4,-.5) node[black,midway,left] {$g_1$};
\end{tikzpicture}
\end{center}

In fact, $\sK^{\can}$ is the \emph{exit-path category} of MacPherson and Treumann \cite{Treumann:exit} with respect to the canonical stratification $\pi$. This example attests to the fine, and yet computable, information that the canonical stratification encodes regarding singularities, and its ability to distinguish between homotopy equivalent spaces like a pinched annulus, an ordinary annulus, and a circle.

Proceeding to the general situation, we can divide this type of analysis into two conceptual steps:
\begin{enumerate}
	\item Perform a traversal of the poset $P$ of simplices of $K$. Starting with the maximal simplices of dimension equal to $n = \dim(K)$ and proceeding down in $P$, label simplices $\sigma$ as \emph{generic} if every simplex $\tau > \sigma$ is generic and the local homology $H_{\ast}(\sigma)$ equals that of an $n$-sphere. This defines the \emph{generic strata} $G \subset P$, which is closed upwards. The complement $P - G$ is then the poset of simplices of a subcomplex $K^1 \subset K$ of smaller dimension. We then repeat the procedure with $K^1$ in place of $K$. Continuing, we obtain a filtration
\[ K = K^0 \supset K^1 \supset K^2 \supset \ldots \supset K^k = \emptyset. \]  
	Define the canonical stratification of $K$ to be the map of posets $\pi: P \to [n]$ given by $\pi(\sigma) = \min\{\dim(K^i) \: | \: \sigma \in K^i \}$.
	\item Perform the ``fiberwise'' localization\footnote{To be more precise, this is the $\Ex^{\infty}_{[n]}$ functor of Remark \ref{rem:localization}.} of $P$ with respect to the map $\pi: P \to [n]$ to obtain the canonical stratified homotopy type $\Pi^{\can}: \sK^{\can} \to [n]$.
\end{enumerate}

This paper is devoted to substantiating and clarifying various aspects pertinent to step (1), deferring the serious study of the invariant $\sK^{\can}$ produced in step (2) to a future work. Our mathematical work is accomplished in \S\S \ref{sec:sheaves}, \ref{sec:deltaStructures}, and \ref{sec:stratification} and is mainly concerned with proving that Algorithm \ref{algorithm} correctly computes the canonical stratification. We then give a pseudocode implementation of Algorithm \ref{algorithm} in \S \ref{sec:st-alg}, which can be read independently from the rest of the paper. Our hope is that this final section will be useful for those looking to apply this algorithm in their own work.
\\

Let us now comment on some features of the theoretical framework in which we will situate our analysis. First, as is already apparent, it is the poset $P$ which is the relevant object for canonical stratification, as opposed to the simplicial complex $K$ and any geometric properties it may otherwise possess. We will thus prove theorems about posets, though we will eventually specialize to the case where the poset is that given by the simplices of an \emph{abstract} simplicial complex (Definition \ref{dfn:deltaStructure}). Second, to systematically reason about the homotopy theory of the various objects involved, we adopt the formalism of $\infty$-categories. Though technically demanding, this theory greatly facilitates the manipulation of homotopy limits and colimits, which recur repeatedly in our work. Third, we will make use of the theory of sheaves on posets. This theory appears because local homology naturally organizes itself into a sheaf on $P$ due to its functoriality in the simplex argument: given $\sigma$ a face of $\tau$, one has an induced homomorphism $H_{\ast}(\sigma) \to H_{\ast}(\tau)$ of local homology groups. We will consider canonical stratification with respect to a homotopical lift of the local $\ZZ$-homology sheaf to the $\infty$-category of spectra, the sheaf $\LL_P$ of Definition \ref{dfn:orientationSheaf}. This yields an \emph{a priori} different definition of canonical stratification (Definition \ref{dfn:strat}) that has better formal properties. We then reconcile this definition with the one described above (Proposition \ref{prp:extendingLocalConstancy} and Remark \ref{rem:Hurewicz}).




\subsection{Main contributions and related work}


The problem of algorithmic determination of the canonical stratification was first studied in \cite{NandaStrat} in the context of regular CW complexes, and subsequently in \cite{BrownWang} in the more general context of posets.\footnote{As discussed in their paper, \cite{BrownWang} works with a slightly different notion of stratification: also see \ref{shriekVsStar}.} Our work continues this line of investigation and culminates in an algorithm (Algorithm \ref{algorithm} and \S \ref{sec:st-alg}) that differs in a few key respects from this existing work. Specifically:

\begin{itemize}
	\item Previous algorithms such as \cite[\S 6.2]{NandaStrat} proceed by checking that various maps in the face poset of a regular CW complex induce quasi-isomorphisms on local (co)homology. In contrast, we prove that for the purposes of the canonical stratification of a simplicial complex, computing maps to be quasi-isomorphisms is superfluous; instead, it suffices to compute the small links (Definition \ref{dfn:link}) of individual simplices to be stable spheres of the appropriate dimension (the aforementioned Proposition \ref{prp:extendingLocalConstancy}).
	\item We also prove that the small links are constrained by Poincar\'{e} duality (Remark \ref{rem:reduceComputationPoincareDuality}), which further reduces the amount of necessary computation.\footnote{For our proofs of these two results, we make use of the combinatorial properties of abstract simplicial complexes and their posets of simplices, though we believe that they more generally extend to face posets of regular CW complexes.}
	\item In particular, for simplicial complexes of dimension $\leq 3$, our algorithm entirely avoids any linear algebraic computation of homology groups (\S \ref{sec:3d-alg}).
\end{itemize}

Along the way, we prove a few results that may be of independent interest, most notably a higher categorical descent result (Theorem \ref{thm:MayerVietoris}) and a result on links (Theorem \ref{thm:shiftingLink}) that may be new in the infinite case.

\subsection{Notation and conventions}

Throughout this paper, we will freely use the formalism of $\infty$-categories as expounded in Lurie's books \cite{HTT,HA}, from which we adopt numerous conventions and pieces of notation. In particular, let us highlight the following:
\begin{itemize}
	\item $[n]$ is the totally ordered set $\{0 < 1 < ... < n\}$. The simplex category $\Delta$ is the category with objects $\{ [n] : n \in \NN \}$ and with morphisms the order preserving maps. $s\Set$ is the category of \emph{simplicial sets}, i.e., functors $\Delta^{\op} \to \Set$.
	\item An \emph{$\infty$-category} is a simplicial set that satisfies the inner horn filling condition, and a \emph{space} is a Kan complex. Note that a	space is thus an \emph{$\infty$-groupoid}, i.e., an $\infty$-category in which every morphism is an equivalence.
	\item $\Cat_{\infty}$ denotes the $\infty$-category of (small) $\infty$-categories, $\sS$ denotes the $\infty$-category of spaces, $\Sp$ denotes the $\infty$-category of spectra, and $\Poset$ denotes the category of posets.
	\item Let $\sS_{\ast} \coloneq \sS^{\ast/}$ be the $\infty$-category of pointed spaces. We have the usual adjunctions
\[ \begin{tikzcd}[row sep=4ex, column sep=4ex, text height=1.5ex, text depth=0.25ex]
\Sigma^{\infty}_+ \colon \sS \ar[shift left=1]{r}{(-)_+} & \sS_{\ast} \ar[shift left=1]{r}{\Sigma^{\infty}} \ar[shift left=1]{l} & \Sp \ar[shift left=1]{l} \colon \Omega^{\infty}
\end{tikzcd}. \]
	\item We have a nerve functor $N: \Poset \to \Cat_{\infty}$, which is fully faithful with essential image given by those $\infty$-categories whose mapping spaces are either empty or contractible. We will typically suppress the extra symbol $N$ and simply refer to $P$ when regarding a poset $P$ as a category.\footnote{We will write $N(P)$ in a few places when we wish to emphasize the simplicial set given by the nerve of $P$.} 
	\item The geometric realization functor $|-|: \Cat_{\infty} \to \sS$ is defined to be the left adjoint of the inclusion $\sS \subset \Cat_{\infty}$ (as opposed to any particular point-set level model). When we write $|P|$ for $P$ a poset, we mean to regard $P$ as a category and then take its geometric realization, i.e., to consider the composite functor $\Poset \xto{N} \Cat_{\infty} \xto{|-|} \sS$.
	\item All categorical constructions are always meant in the $\infty$-categorical sense. For example, limits and colimits are necessarily homotopy limits and colimits.
	\item Limit and colimits involving posets are always regarded as being computed in $\Cat_{\infty}$. Note that limits of posets are computed the same in $\Poset$ or $\Cat_{\infty}$. However, our convention has force when computing colimits; for example, if $p: I \to \Poset$ is the constant functor at the terminal poset, then $\varinjlim (N \circ p) \simeq |I|$ is a space, which is not generally equivalent to a poset.
    \item Following \cite{HA}, we denote the smash product of spectra by $\otimes$ rather than $\wedge$.
\end{itemize}

\subsection{Acknowledgements}

We thank Elden Elmanto and the anonymous referee for a number of helpful comments on earlier versions of this article. We are also grateful to Jonas Frey for locating a mistake in our previous formulation of Corollary \ref{cor:SheavesAsFunctors} and to Ko Aoki for pointing out a mistake in our previous proof of Proposition \ref{prp:SheavesAreRKE}. J.S. was supported by NSF grant DMS-1547292 and by the Deutsche Forschungsgemeinschaft (DFG, German Research Foundation) under Germany’s Excellence Strategy EXC 2044–390685587, Mathematics Münster: Dynamics–Geometry–Structure.
Code for the dimension $\leq 3$ version of the algorithm is provided in a repository at \url{https://github.com/ColfaxResearch/CanonicalStratification}.

\section{Sheaves on posets} \label{sec:sheaves}


Let $P$ be a poset. 

\begin{dfn} Let $Q \subset P$ be a subposet. Then $Q$ is a \emph{cosieve} if it is closed upwards in $P$, i.e., for every $x \in Q$, if $y \geq x$ in $P$ then $y \in Q$. Dually, $Q$ is a \emph{sieve} if it is closed downwards in $P$, i.e., for every $x \in Q$, if $y \leq x$ in $P$ then $y \in Q$.

Let $\Op(P)$, resp. $\Cl(P)$ denote the poset of cosieves, resp. sieves of $P$, with the partial order defined by inclusion.
\end{dfn}

One may endow the set $P$ with the structure of a topological space by declaring the open sets to be the cosieves of $P$. Alternatively, one can view $P$ as a category with objects given by the elements $x \in P$ and morphisms defined by the relation that there exists a unique morphism $x \to y$ if and only if $x \leq y$. When we reason about the homotopical properties of posets, it will be this second, categorical perspective that predominates. Nonetheless, the intuition afforded by topological spaces has its uses, most notably in understanding sheaves on posets. In this introductory section, we will explain how a few standard notions regarding sheaves on topological spaces translate to the setting of posets. We will also introduce and study the ($\Sp$-valued) local homology sheaf $\LL_P$ of a poset (Definition \ref{dfn:orientationSheaf}), in preparation for the stratification algorithm of \S \ref{sec:stratification}.
\\

Let $\sC$ be an $\infty$-category with all limits. We first recall the definition of a $\sC$-valued sheaf on a topological space.

\begin{dfn} Suppose $X$ is a topological space and let $\Op(X)$ be the poset of open subsets of $X$. We define a Grothendieck topology (\cite[Def.~6.2.2.1]{HTT}) on $\Op(X)$ as follows: for every open subset $U \subset X$, a sieve $J \subset \Op(X)_{\leq U}$ is said to be a \emph{covering sieve} if and only if the union over all open sets contained in $J$ equals $U$. Then a functor $F: \Op(X)^{\op} \to \sC$ is a \emph{$\sC$-valued sheaf} if and only if for every covering sieve $J \subset \Op(X)_{\leq U}$, the natural map
\[ F(U) \xto{\sim} \varprojlim_{V \in J^{\op}} F(V)  \]
is an equivalence in $\sC$.

Let $\Shv_{\sC}(X) \subset \Fun(\Op(X)^{\op}, \sC)$ denote the full subcategory of $\sC$-valued sheaves on $X$.
\end{dfn}

Suppose that $X \to P$ is a continuous map of topological spaces with $P$ a poset topologized as indicated above. Then, modulo some technical conditions, it is a theorem of Lurie (\cite[Thm.~A.9.3]{HA}), extending work of MacPherson and Treumann \cite{Treumann:exit}, that there exists an $\infty$-category $\Exit_P(X)$ such that we have an equivalence of $\infty$-categories
\[ \Fun(\Exit_P(X), \sC) \simeq \Shv^{P\text{-cnstr}}_{\sC}(X) \]
between $\sC$-valued functors on $\Exit_P(X)$ and $P$-constructible $\sC$-valued sheaves on $X$.\footnote{A sheaf $F$ on $X$ is \emph{$P$-constructible} if it is locally constant when restricted to each fiber $X_p$ of the map $X \to P$.} In the degenerate case where $X = P$, $\Exit_P(X) \simeq P$ with $P$ regarded as a category via its nerve, and the condition of $P$-constructibility is automatic. We therefore obtain an equivalence
\[ \Fun(P, \sC) \simeq \Shv_{\sC}(P). \]

Unfortunately, the equivalence furnished by \cite[Thm.~A.9.3]{HA} is too inexplicit for our purposes. Our first goal is to give a direct proof of this equivalence in the case that $P$ is finite-dimensional (Definition \ref{dfn:dimension}). Observe that we have a full and faithful inclusion of posets $P \subset \Op(P)^{\op}$ given by sending an object $x \in P$ to the cosieve $P_{\geq x}$.

\begin{prp} \label{prp:SheavesAreRKE} A functor $F: \Op(P)^{\op} \to \sC$ is a $\sC$-valued sheaf if $F$ is a right Kan extension of its restriction to $P$.
\end{prp}
\begin{proof} Suppose that $F$ is a right Kan extension of $F|_{P}$. Let $J \subset \Op(P)_{\leq S}$ be any covering sieve for the Grothendieck topology on $\Op(P)$; this means that $J$ is a sieve in $\Op(P)_{\leq S}$ whose elements $\{S_\alpha \}$ form a cover of $S$. The inclusion $J^{\op} \subset \Op(P)^{\op}$ extends to a functor $(J^{\op})^{\lhd} \to \Op(P)^{\op}$ that sends the cone point to $S$, and for $F$ to be a sheaf it suffices to verify that the restriction of $F$ to $(J^{\op})^{\lhd}$ is a limit diagram. To prove this, we apply a limit decomposition result from \cite[\S 4.2.3]{HTT}. Define a functor
\[ H: J \to s\Set_{/N(S)}  \]
which sends $T \in J$ to $N(T) \subset N(S)$. We claim that $H$ satisfies the hypotheses of \cite[Prop.~4.2.3.8]{HTT}. By \cite[Rem.~4.2.3.9]{HTT}, it suffices to check that for any non-degenerate $n$-chain
$$\sigma = \left[ x_0 < x_1 < ... < x_n \right]$$
in $S$, the subposet $J_\sigma$ on objects $\{ T \in J : \sigma \subset T \}$ is weakly contractible. Using that the constituent elements $\{ S_\alpha \}$ of $J$ cover $S$, there exists some $S_\alpha$ such that $x_0 \in S_\alpha$, hence $\sigma \subset S_{\alpha}$ because $S_\alpha$ is a cosieve. Therefore, $J_\sigma$ is nonempty. Moreover, given two objects $T, T' \in J_{\sigma}$, we have that the intersection $T \cap T'$ is the product in $J_{\sigma}$. As a nonempty category that admits binary products, $J_{\sigma}$ is weakly contractible; indeed, recalling the standard argument, the adjunction $\left( \adjunct{\Delta}{J_{\sigma}}{J_{\sigma} \times J_{\sigma}}{\times} \right)$ implies that $|J_{\sigma}| \times |J_{\sigma}| \simeq |J_{\sigma}|$, hence $|J_{\sigma}| \simeq \ast$. Alternatively, one can observe that $J_{\sigma}$ is cosifted, hence weakly contractible.

Now by the dual of \cite[Cor.~4.2.3.10]{HTT} applied to the functor $F|_S: S \to \sC$, we see that
\[  \varprojlim_{x \in S} F|_S(x) \xto{\sim} \varprojlim_{T \in J^{\op}} \varprojlim_{x \in T} F|_T(x). \]
Because $F$ is a right Kan extension of $F|_P$, this map is identified with
\[ F(S) \xto{\sim}  \varprojlim_{T \in J^{\op}} F(T), \]
and the claim is proven.
\end{proof}

\begin{rem}
We comment on the use and meaning of \cite[Cor.~4.2.3.10]{HTT}. Suppose given an $\infty$-category $\sC$, a small $\infty$-category $\sK$, and a functor $p: \sK \to \sC$. Then given \emph{any} small $\infty$-category $\sI$ and functor $p_{\bullet}: \sI \to (\Cat_{\infty})_{/\sC}$ such that $\colim_{\sI} p_{\bullet} \simeq p$ and each $p_i: \sK_i \to \sC$ admits a limit $x_i \in \sC$, we have a natural equivalence
\begin{equation} \label{eq:limit_decomposition}
\varprojlim_{\sK} p \simeq \varprojlim_{i \in \sI^{\op}} \varprojlim_{\sK_i} p_i.
\end{equation}
Indeed, this follows essentially from the ``slice-slice'' adjunction
\[ \adjunct{\sC_{(-)/}}{(\Cat_{\infty})_{/\sC}}{(\Cat_{\infty})^{\op}_{/\sC}}{\sC_{/(-)}}; \]
see \cite[Thm.~8.1]{shah2022parametrized} for a proof.\footnote{The reference actually establishes a generalization to the setting of parametrized (co)limits with respect to a fixed base $\infty$-category $\sT$. Specializing to the case $\sT = \ast$ then obtains the equivalence \eqref{eq:limit_decomposition}.} The results of \cite[\S4.2.3]{HTT} then amount to an implementation of this observation at the level of simplicial sets, whence the conditions that appear in \cite[Prop.~4.2.3.8]{HTT}.
\end{rem}

\begin{prp} \label{prp:Hypercompleteness} Suppose $P$ is a finite-dimensional poset and let $\phi: F \to G$ be a morphism of $\sC$-valued sheaves on $P$ such that for all $x \in P$, the induced map $\phi_x: F(P_{\geq x}) \to G(P_{\geq x})$ is an equivalence in $\sC$. Then $\phi$ is an equivalence in $\Shv_{\sC}(P)$.
\end{prp}
\begin{proof} We proceed by induction on the dimension $n = \dim(P)$ of $P$. The claim is obvious if $n = 0$, i.e., $P$ is a discrete set, so suppose for the inductive hypothesis that $n >0$ and we have proven the claim for all posets $Q$ of dimension less than $n$. Let $U \subset P$ be a cosieve and consider the covering sieve $J \subset \Op(X)_{\leq U}$ generated by $\{ P_{\geq x} : x \in U \}$. To show that $\phi_U: F(U) \to G(U)$ is an equivalence, it suffices to show that $\phi_V: F(V) \to G(V)$ is an equivalence for all $V \in J$, i.e., for all $V$ contained in $P_{\geq x}$ for some $x \in U$. Either $V = P_{\geq x}$, in which case $\phi_V$ is an equivalence by assumption, or $x \notin V$ and hence $\dim(V) < n$, in which case $\phi_V$ is an equivalence by the inductive hypothesis. We conclude that $\phi$ is an equivalence.
\end{proof}

\begin{cor} \label{cor:SheavesAsFunctors} Right Kan extension along the inclusion $j: P \sub \Op(P)^{\op}$ implements a fully faithful embedding
\[ j_{\ast}: \Fun(P,\sC) \subset \Shv_{\sC}(P)  \]
that is right adjoint to the restriction $j^{\ast}$. Furthermore, if $P$ is finite-dimensional, then $j_{\ast}$ is an equivalence of $\infty$-categories.
\end{cor}
\begin{proof} Consider the adjunction of functor $\infty$-categories
$$ \adjunct{j^{\ast}}{\Fun(\Op(P)^{\op}, \sC)}{\Fun(P,\sC)}{j_{\ast}}$$
given by restriction $j^{\ast}$ and right Kan extension $j_{\ast}$. Since $j$ is fully faithful, $j_{\ast}$ is fully faithful by \cite[Prop.~4.3.2.15]{HTT}. Proposition \ref{prp:SheavesAreRKE} then shows that the essential image of $j_{\ast}$ is contained in the full subcategory $\Shv_{\sC}(P)$ and $j^{\ast} \dashv j_{\ast}$ restricts to the claimed adjunction. For the second assertion, Proposition \ref{prp:Hypercompleteness} implies that the unit map $F \to j_{\ast} j^{\ast} F$ is an equivalence for all $\sC$-valued sheaves $F$, hence $j_{\ast}$ is an equivalence. 
\end{proof}

\begin{rem} In the setting of $1$-categories, such a result has been previously obtained by Justin Curry \cite[\S 4.2.2]{CurryThesis} without finite-dimensionality hypotheses on $P$; see also \cite[Prop. B.6.4]{LurieUltra}. Corollary \ref{cor:SheavesAsFunctors} is not a precise analogue of this $1$-categorical result due to potential difficulties involving the failure of the $\infty$-topos $\Shv_{\sS}(P)$ to be hypercomplete. Indeed, Jonas Frey has informed us of a counterexample to the equivalence of Corollary \ref{cor:SheavesAsFunctors} for an infinite-dimensional poset, which he attributed to Charles Rezk and Mathieu Anel. This counterexample has also appeared in Aoki's work as \cite[Exm.~A.13]{aoki2020tensor}.

Aoki has since proven that for a general poset $P$, the embedding $j_{\ast}$ identifies $\Fun(P,\sS)$ with the full subcategory of $\Shv_{\sS}(P)$ on the hypercomplete sheaves \cite[Exm.~A.11]{aoki2020tensor} (which we stated as a conjecture in an earlier version of this paper). This is also recorded as \cite[Exm.~3.12.15]{Exodromy}, which cites Aoki for this fact.
\end{rem}

\begin{dfn} In view of Corollary \ref{cor:SheavesAsFunctors}, we will interchangeably refer to functors $F: P \to \sC$ as sheaves. If $F$ sends every morphism in $P$ to an equivalence in $\sC$, we say that $F$ is \emph{locally constant}. If $F$ is moreover equivalent to a constant functor, we say that $F$ is \emph{constant}.
\end{dfn}

We now define the sheaf of central interest in this paper. Given a poset $P$, let $c: P \to \Poset$ be the constant functor at $P$, and let $F: P \to \Poset$ be the functor $x \mapsto P - P_{\geq x}$. We have a natural transformation $\theta: F \Rightarrow c$ given objectwise by inclusion.

\begin{dfn} \label{dfn:orientationSheaf} Let $\LL_P^{\sS}$ be the cofiber of $|\theta|$ regarded as valued in pointed spaces
\[ \LL_P^{\sS}: P \to \sS_{\ast} \: , \quad x \mapsto |P|/|P-P_{\geq x}|. \]

We define the \emph{local homology sheaf} $\LL_P: P \to \Sp$ to be $\Sigma^{\infty }\LL^{\sS}_P$.\footnote{The adjective ``homology'' should be interpreted here with respect to coefficients given by the sphere spectrum.} Furthermore, given a ring $k$, we define the \emph{local $k$-homology sheaf} $\LL_{P}^k: P \to \categ{D}(k)$ by postcomposition of $\LL_P$ with the base-change functor $- \otimes H k : \Sp \to \categ{D}(k)$, where $\categ{D}(k)$ is the unbounded derived category of $k$-modules (viewed as an $\infty$-category).\footnote{Recall that this base-change functor is indeed given by smashing with the Eilenberg-MacLane spectrum $H k$ under the equivalence $\categ{D}(k) \simeq \Mod_{H k}$.}
\end{dfn}

\begin{rem} We have a point-set level model of $\LL_P^k$ as a functor valued in the category $\Ch_k$ of chain complexes of $k$-modules, given by $\LL_P^k(x) \coloneq C_{\ast}(P, P- P_{\geq x}; k)$. We will not emphasize this perspective on the local $k$-homology sheaf in this paper.
\end{rem}

\begin{prp} \label{prp:ExtendedOrientationSheaf} Suppose that $P$ is a finite poset.
\begin{enumerate} \item The functor $\widetilde{\LL_P}: \Op(P)^{\op} \to \Sp$ defined by
 \[ \widetilde{\LL_P}(S) \coloneq \Sigma^\infty(|P|/|P-S|) \]
is a sheaf.

\item For any ring $k$, the functor $\widetilde{\LL_P^k}: \Op(P)^{\op} \to \categ{D}(k)$ defined by 
\[ \widetilde{\LL_P^k}(S) \coloneq C_\ast(P,P-S; k) \]
is a sheaf.
\end{enumerate}
\end{prp}
\begin{proof} Because $P$ is finite, (2) follows from (1) in view of the base-change functor $\Sp \to \categ{D}(k)$ preserving finite limits. For (1), again using that $P$ is finite, it suffices to check that
\begin{enumerate} \item[(a)] $\widetilde{\LL_P}(\emptyset) \simeq 0$.
\item[(b)] For every cover of a cosieve $S$ by two cosieves $S_0$ and $S_1$, we have a pullback square
\[ \begin{tikzcd}[row sep=4ex, column sep=4ex, text height=1.5ex, text depth=0.25ex]
\widetilde{\LL_P}(S) \ar{r} \ar{d} & \widetilde{\LL_P}(S_0) \ar{d} \\
\widetilde{\LL_P}(S_1) \ar{r} & \widetilde{\LL_P}(S_0 \cap S_1).
\end{tikzcd} \]
\end{enumerate}
Indeed, given any covering sieve $J \subset \Op(P)_{\leq S}$ with maximal elements $\{ S_0, S_1 \}$, it is easy to check that the inclusion of $\{ S_0 \leftarrow S_0 \cap S_1 \rightarrow S_1 \}$ into $J$ is cofinal; moreover, a simple inductive argument extends the scope of (b) to a covering of $S$ by $n$ cosieves.

(a) holds by definition. For (b), using that $\Sigma^\infty$ sends pushout squares to pullback squares, it suffices to show that we have a pushout square of spaces

\[ \begin{tikzcd}[row sep=4ex, column sep=4ex, text height=1.5ex, text depth=0.25ex]
\left|P- S \right| \ar{r} \ar{d} & \left| P - S_0 \right| \ar{d} \\
\left| P-S_1 \right| \ar{r} & \left| P - (S_0 \cap S_1) \right|,
\end{tikzcd} \]

which follows from Theorem \ref{thm:MayerVietoris}.\footnote{Theorem \ref{thm:MayerVietoris} can also be used to directly handle the case of a cover of $S$ by $n$ cosieves.}
\end{proof}

\begin{cor} \label{cor:extendedOrientationSheaf} Under the equivalence of \ref{cor:SheavesAsFunctors}, $\widetilde{\LL_P}$ (resp. $\widetilde{\LL^k_P}$) corresponds to $\LL_P$ (resp. $\LL_P^k$).
\end{cor}
\begin{proof} Together with Proposition \ref{prp:ExtendedOrientationSheaf}, this follows immediately from the observation that $(\widetilde{\LL_P})|_P = \LL_P$.
\end{proof}

\subsection{Descent over posets}

Consider a nice topological space $X$ and a vector bundle $V \to X$. The basic idea of \emph{descent} is that given a suitable open cover $\{ U_i \}$ of $X$, we can reconstruct $V \to X$ from the pulled-back vector bundles $V_i \coloneq V \times_{X} U_i \to U_i$ together with the ``gluing'' data of isomorphisms $\phi_{i j}: V_{i} \times_{U_i} (U_i \cap U_j) \xto{\cong} V_j \times_{U_j} (U_i \cap U_j)$ that satisfy appropriate compatibility (``cocycle'') conditions. In this subsection, we consider descent in a higher categorical setting. We replace the topological space $X$ by a poset $P$, the vector bundle $V \to X$ by a functor $\sC \to P$ with $\sC$ an $\infty$-category, and the cover $\{U_i\}$ of $X$ by a sieve or cosieve covering $\{P_i\}$ of $P$. Our main result is the formula of Theorem \ref{thm:MayerVietoris}. In the sequel, we will only need the case where $\sC = P$. However, we have decided to phrase our results at this level of generality so as to better expose the underlying ideas. 

\begin{prp} \label{prp:SieveCovering} Let $\sC$ be an $\infty$-category and let $\pi: \sC \to P$ be a functor. Regard $P$ as a subposet of the poset $\Cl(P)$ of sieves of $P$ via $x \mapsto P_{\leq x}$. Then the functor
\[ F: \Cl(P) \to \Cat_{\infty}, \quad Z \mapsto \sC \times_P Z \]
is a left Kan extension of its restriction to $P$.

Dually, regard $P^{\op}$ as a subposet of the poset $\Op(P)$ of cosieves of $P$ via $x \mapsto P_{\geq x}$. Then the functor
\[ G: \Op(P) \to \Cat_{\infty}, \quad U \mapsto \sC \times_P U \]
is a left Kan extension of its restriction to $P^{\op}$.
\end{prp}
\begin{proof} For the claim about $F$, we need to check that for any sieve $Z \subset P$, 
\[ \varinjlim\limits_{x \in Z} \sC \times_P P_{\leq x} \xto{\sim} \sC \times_P Z.\]
Replacing $\sC$ by $\sC \times_P Z$, we may suppose that $Z = P$. Then the claim is a consequence of the following general fact about a categorical fibration\footnote{Note that any functor $\sC \to P$ is necessarily a categorical fibration by \cite[Prop.~2.3.1.5]{HTT} and every equivalence in $P$ being an identity morphism.} $\pi: \sC \to \sB$ (where the base $\sB$ is now taken to be an arbitrary $\infty$-category): the colimit of the functor
\[ \sB \to \Cat_{\infty}, \quad x \mapsto \sC \times_{\sB} \sB^{/x} \]
is equivalent to $\sC$. To prove this, we use that this functor classifies the cocartesian fibration $\ev_1: \sC \times_{\sB} \sO(\sB) \to \sB$ (where $\sO(\sB)$ is notation for the arrow $\infty$-category $\Fun(\Delta^1, \sB)$). Therefore, an explicit model for its colimit is given by inverting the $\ev_1$-cocartesian edges in $\sC \times_{\sB} \sO(\sB)$ (cf. \cite[Cor.~3.3.4.3]{HTT}). Let $\cE$ be the collection of the $\ev_1$-cocartesian edges, so an edge
\[ \left( \begin{tikzcd}[row sep=4ex, column sep=4ex, text height=1.5ex, text depth=0.25ex]
c_0 \ar{r}{\alpha} & c_1
\end{tikzcd}, \quad
\begin{tikzcd}[row sep=4ex, column sep=4ex, text height=1.5ex, text depth=0.25ex]
\pi(c_0) \ar{r} \ar{d} & \pi(c_1) \ar{d} \\
t_0 \ar{r} & t_1
\end{tikzcd}
 \right) \]
belongs to $\cE$ if and only if $\alpha$ is an equivalence.

Regard $(\sC \times_{\sB} \sO(\sB), \cE)$ as a marked simplicial set. To complete the proof, it suffices to show that the projection functor $\pr_{\sC}: (\sC \times_{\sB} \sO(\sB), \cE) \to \sC^{\sim}$ is an equivalence in the marked model structure on $s\Set^+$ of \cite[\S 3]{HTT}, where $\sC^{\sim}$ denotes that we mark $\sC$ with its equivalences. For this, we observe that the identity section $\iota: \sC \to \sC \times_\sB \sO(\sB)$ is left adjoint to $\pr_{\sC}$, with $\pr_{\sC} \circ \iota = \id$ and counit map $\epsilon$
\[ \begin{tikzcd}[row sep=4ex, column sep=6ex, text height=1.5ex, text depth=0.25ex]
\{ 0 \} \times  \sC \times_\sB \sO(\sB) \ar{d} \ar[bend left=20]{rd}{\iota \circ \pr_C} \\
\Delta^1 \times \sC \times_\sB \sO(\sB) \ar{r}{\epsilon} & \sC \times_\sB \sO(\sB) \\
\{ 1 \} \times  \sC \times_\sB \sO(\sB) \ar{u} \ar[bend right=20]{ru}[swap]{\id}
\end{tikzcd} \]
defined such that on objects $(c,f: \pi(c) \to t)$, $\epsilon$ is given by the $\ev_1$-cocartesian edge
\[ \left( \begin{tikzcd}[row sep=4ex, column sep=4ex, text height=1.5ex, text depth=0.25ex]
c \ar{d}[swap]{=} \\
c
\end{tikzcd}, \quad
\begin{tikzcd}[row sep=4ex, column sep=4ex, text height=1.5ex, text depth=0.25ex]
\pi(c) \ar{r}{=} \ar{d}{=} & \pi(c) \ar{d}{f} \\
\pi(c) \ar{r}{f}  & t
\end{tikzcd} \right). \]
$\epsilon$ thus furnishes a marked homotopy between $\id$ and $\pr_C \circ \iota$ as self maps of $(\sC \times_{\sB} \sO(\sB), \cE)$, so we conclude that $\pr_{\sC}$ (as well as $\iota$) is a marked homotopy equivalence, \textit{a fortiori} a weak equivalence in $s\Set^+$.

Finally, a dual argument handles the claim about $G$.
\end{proof}

\begin{thm} \label{thm:MayerVietoris} Let $\sC$ be an $\infty$-category and let $\pi: \sC \to P$ be a functor. Let $P_0, ..., P_n$ be subposets of $P$ which cover $P$. Suppose either that every $P_i$ is a sieve or that every $P_i$ is a cosieve. Then we have an equivalence of $\infty$-categories
\[ \varinjlim\limits_{\emptyset \neq I \subset [n]} \left( \left( \bigcap_{i \in I} P_i \right) \times_P \sC \right) \xto{\sim} \sC \]
where the colimit is taken over the poset of nonempty subsets of $[n]$, ordered by reverse inclusion.
\end{thm}
\begin{proof} We may suppose that every $P_i$ is a sieve, the cosieve case following from a dual argument. Let $\sd([n])$ be the poset of nonempty subsets of $[n]$, ordered by inclusion\footnote{Viewing $[n]$ itself as a poset, this is the usual barycentric subdivision of $[n]$.}. Define a map of sets $\phi: P \to \sd([n])^{\op}$ by
\[ \phi(x) = \{ i \in [n]: x \in P_i \}. \]
If $x \leq y$, then for every $i \in [n]$ such that $y \in P_i$, we necessarily have $x \in P_i$, so $\phi(y) \subset \phi(x)$. Therefore, $\phi$ is order-preserving, so $\phi$ is a functor. For $I \in \sd([n])^{\op}$, note that
\[ P \underset{\sd([n])^{\op}}{\times} (\sd([n])^{\op})_{\leq I} = \bigcap_{i \in I} P_i. \]

Applying Proposition \ref{prp:SieveCovering} to the composite functor $\phi \circ \pi$, we deduce the claim.
\end{proof}


\subsection{Recollement of sheaves on posets}


Having defined the local homology sheaf $\LL_P$, it is natural to ask about the functoriality properties of $\LL_{(-)}$ in the poset argument $P$. In this subsection, we study the functoriality of $\LL_{(-)}$ in the situation of a sieve inclusion $i: Q \to P$ with $P$ a finite poset (Proposition \ref{prp:shriekFunctoriality}). In the sequel, this material will only be used to clarify the difference between our notion of canonical stratification and other possible approaches (see \ref{shriekVsStar} below).

To properly articulate the relation between $\LL_Q$ and $\LL_P$, we will use the formalism of \emph{recollement} of sheaves. Suppose $X$ is a topological space, $i: Z \to X$ is the inclusion of a closed subspace and $j: U = X - Z \to X$ is the inclusion of its open complement. Then we have various functors between the categories of sheaves of sets on $X$, $Z$, and $U$. For example, we have the pushforward-pullback adjunctions
\[ \adjunct{i^\ast}{\Shv(X)}{\Shv(Z)}{i_\ast}, \quad \adjunct{j^\ast}{\Shv(X)}{\Shv(U)}{j_\ast}. \]
We also have a ``gluing'' relation: given a sheaf $\cF$ on $X$, we have a pullback square
\[ \begin{tikzcd}[row sep=4ex, column sep=4ex, text height=1.5ex, text depth=0.25ex]
\cF \ar{r} \ar{d} & i_\ast i^\ast \cF \ar{d} \\
j_\ast j^\ast \cF \ar{r} & i_\ast i^\ast j_\ast j^\ast \cF 
\end{tikzcd} \]
where the right vertical arrow is induced by a certain canonical map $i^\ast \cF \to i^\ast j_\ast j^\ast \cF$. In fact, the datum of a sheaf $\cF$ on $X$ amounts to the data of sheaves $\cF_Z$ on $Z$, $\cF_U$ on $U$, and a map of sheaves $\cF_Z \to i^\ast j_\ast \cF_U$. In this situation, we say that $\Shv(X)$ is a \emph{recollement} of $\Shv(U)$ and $\Shv(Z)$ (\cite[Def.~A.8.1]{HA}). See \cite[\S A.8]{HA}, \cite{BarwickGlasmanNoteRecoll}, or \cite[\S 2]{shah2022recollements} for general references on the theory of recollements in the setting of $\infty$-categories.

\begin{nul} We now replace $X$ by the poset $P$, $Z \to X$ by a sieve inclusion $i: Q \to P$, $U \to X$ by the complementary cosieve inclusion $j: R = P-Q \to P$, and the category of sets by a stable $\infty$-category $\sC$ with all limits and colimits (e.g., $\Sp$ or $\categ{D}(\ZZ)$). We want to show that $\Fun(P,\sC)$ is a recollement of $\Fun(Q, \sC)$ and $\Fun(R, \sC)$. To accomplish this task, we will apply \cite[Prop.~A.8.7]{HA} to a certain correspondence $\pi: \sM \to \Delta^1$. Let $p: P \to \Delta^1$ be the map of posets which sends $x \in Q$ to $0$ and $y \in R$ to $1$. Define $\pi: \sM \to \Delta^1$ to be the simplicial set given by the formula:
\begin{itemize}
	\item[($\ast$)] Maps of simplicial sets $K \to \sM$ over $\Delta^1$ are in bijection with maps $K \times_{\Delta^1} P \to \sC$, i.e.
	\[ \Hom_{/\Delta^1}(K,\sM) = \Hom(K \times_{\Delta^1} P, \sC). \]
\end{itemize}
Then $\sM$ is an $\infty$-category by \cite[Prop.~B.4.5]{HA} applied to the flat inner fibration $p: P \to \Delta^1$. Moreover, we have an identification of the fibers and sections of $\sM$:
 \[ \sM_0 \simeq \Fun(Q,\sC), \: \sM_1 \simeq \Fun(R, \sC), \: \Fun_{/\Delta^1}(\Delta^1, \sM) \simeq \Fun(P, \sC) \]
and $\pi$ is a cartesian fibration classified by the functor $(\Delta^1)^{\op} \to \Cat_{\infty}$ given by the composition
\[ \Fun(R, \sC) \xto{j_{\ast}} \Fun(P, \sC) \xto{i^{\ast}} \Fun(Q, \sC) \]
where $j_{\ast}$ denotes right Kan extension along $j$ and $i^\ast$ denotes restriction along $i$. Note that $i^\ast j_{\ast}$ is an exact functor because $\sC$ is stable. Therefore, $\sM$ is a left-exact correspondence (\cite[Def.~A.8.6]{HA}), which yields our desired recollement.
\end{nul}

In the situation of a recollement of stable $\infty$-categories, we have the diagram of adjunctions
\[ \begin{tikzcd}[row sep=4ex, column sep=8ex, text height=1.5ex, text depth=0.25ex]
\Fun(R, \sC) \ar[shift left=4]{r}{j_!} \ar[shift right=4]{r}[swap]{j_\ast} & \Fun(P, \sC) \ar{l}[description]{j^\ast} \ar[shift left=4]{r}{i^\ast} \ar[shift right=4]{r}[swap]{i^!} & \Fun(Q, \sC) \ar{l}[description]{i_\ast}
\end{tikzcd} \]
where $i^!$ is the \emph{exceptional inverse image} or ``shriek pullback'' functor, which fits into a fiber sequence of functors
\begin{equation} \label{shriekFiberSequence} i^! \to i^\ast \to i^\ast j_\ast j^\ast.
\end{equation}

See \cite[Obs.~2.18]{shah2022recollements}. Now let $\sC = \Sp$ and consider the local homology sheaf $\LL_P \in \Fun(P, \Sp)$. For all objects $x \in Q$, we have morphisms $\Sigma^\infty (|Q|/|Q- Q_{\geq x}|) \to \Sigma^{\infty} (|P|/|P - P_{\geq x}|)$ induced by the various inclusions of posets that assemble into a natural transformation $\LL_Q \to i^\ast \LL_P = (\LL_P)|_Q$. 

\begin{prp} \label{prp:shriekFunctoriality} Suppose $P$ is a finite poset. Then the map $\LL_Q \to i^\ast \LL_P$ factors through an equivalence $\LL_Q \xto{\sim} i^! \LL_P$.
\end{prp}
\begin{proof} We first compute $j_\ast ((\LL_P)|_R)$ on objects $x \in Q$. By the pointwise formula for right Kan extension, we have that
\[ (j_\ast ((\LL_P)|_R))(x) \simeq \varprojlim_{y \in R_{\geq x} \coloneq (R \times_P P_{\geq x})} (\LL_P)(y). \]

Since $R_{\geq x}$ is a cosieve in $P$, by Corollary \ref{cor:extendedOrientationSheaf} (which requires $P$ to be finite) the limit can be identified with
\[\widetilde{\LL_P}(R_{\geq x}) \coloneq \Sigma^\infty(|P|/|P- R_{\geq x}|), \]
and a chase of the definitions shows that the unit map $\LL_P \to j_\ast j^\ast \LL_P$ evaluated on $x$ is the map $\Sigma^\infty(|P|/|P- P_{\geq x}|) \to \Sigma^\infty(|P|/|P- R_{\geq x}|)$ induced by the inclusion $P-P_{\geq x} \subset P-R_{\geq x}$.

Upon passage to cofibers, the sequence of inclusions $P - P_{\geq x} \subset P - R_{\geq x} \subset P$ yields the fiber sequence
\[ \Sigma^\infty (|P - R_{\geq x}| / |P - P_{\geq x} |) \to \Sigma^\infty(|P|/|P- P_{\geq x}|) \to \Sigma^\infty(|P|/|P- R_{\geq x}|), \]
which calculates the fiber term $(i^! \LL_P)(x)$. Then applying Theorem \ref{thm:MayerVietoris} to the square of sieve inclusions
\[ \begin{tikzcd}[row sep=4ex, column sep=4ex, text height=1.5ex, text depth=0.25ex]
(P - (P_{\geq x} \cup R)) = (Q - Q_{\geq x})  \ar{r} \ar{d} & P - R = Q \ar{d} \\
P - P_{\geq x} \ar{r} & (P- (R \cap P_{\geq x})) = (P - R_{\geq x})
\end{tikzcd} \]
we see that 
\[ \begin{tikzcd}[row sep=4ex, column sep=4ex, text height=1.5ex, text depth=0.25ex]
\Sigma^\infty_+ |Q - Q_{\geq x}| \ar{r} \ar{d} & \Sigma^\infty_+ |Q| \ar{d} \\
\Sigma^\infty_+ |P - P_{\geq x}| \ar{r} & \Sigma^\infty_+ |P - R_{\geq x}|
\end{tikzcd} \]
is a pushout square, so the induced map of cofibers
\[ (\LL_Q)(x) = \Sigma^\infty (|Q|/|Q - Q_{\geq x}|) \to (i^! \LL_P)(x) = \Sigma^\infty (|P - R_{\geq x}| / |P - P_{\geq x} |)\]
is an equivalence (clearly natural in $x \in Q$), proving the claim.
\end{proof}

\section{$\delta$-structures on posets} \label{sec:deltaStructures}
To proceed with our program for stratifying posets $P$ via the local homology sheaf $\LL_P$, we need to constrain $P$ so that the notion of local neighborhood used in the definition of $\LL_P$ is reasonable (for example, invariant under subdivision). To do this, we will restrict our attention to those posets which arise as the poset of simplices of a simplicial complex; we call such posets \emph{$\delta$-admissible} (Definition \ref{dfn:deltaStructure}). To explicate the relevance of this condition and how it constrains the behavior of $\LL_P$, it will be convenient to introduce the formalism of discrete cartesian fibrations over $\Delta^{\inj}$.

\begin{ntn} Let $\Delta^{\inj}$ denote the subcategory of the simplex category $\Delta$ with the same objects and morphisms $[k] \to [l]$ taken to be the injective maps of totally ordered sets.
\end{ntn}

\begin{rec} \label{rec:cartesianFibration} Let us briefly recall the notion of \emph{cartesian fibration} that we use to formulate Definition \ref{dfn:deltaStructure}; see \cite[\S 2.4]{HTT} or \cite[\S 3]{BarShah} for a more systematic discussion. Let $\sX$ and $\sC$ be categories.\footnote{The restriction to ordinary categories is only for expository purposes: we could also take $\sX$, $\sC$ to be $\infty$-categories. In that case, we should also demand that $F: \sX \to \sC$ is an inner fibration and consider mapping spaces instead of hom sets.} Given a functor $F: \sX \to \sC$, we say that a morphism $f: y \to x$ in $\sX$ is a \emph{$F$-cartesian edge} if it enjoys the following lifting property: for every object $z \in \sX$, the commutative square of hom sets
\[ \begin{tikzcd}[row sep=4ex, column sep=4ex, text height=1.5ex, text depth=0.25ex]
\Hom_{\sX}(z,y) \ar{r}{f_{\ast}} \ar{d}{F} & \Hom_{\sX}(z,x) \ar{d}{F} \\
\Hom_{\sC}(F(z), F(y)) \ar{r}{F(f)_{\ast}} & \Hom_{\sC}(F(z), F(x)) 
\end{tikzcd} \]
is a pullback square. We say that $F$ is a \emph{cartesian fibration} if for every morphism $\alpha: d \to c$ in $\sC$ and $x \in \sX$ such that $F(x) = c$, there exists a cartesian edge $f: y \to x$ with $F(f) = \alpha$. $F$ is moreover \emph{discrete} if its fibers are all equivalent to sets.

Dually, $F$ is a cocartesian fibration if $F^{\op}: \sC^{\op} \to \sD^{\op}$ is a cartesian fibration.
\end{rec}

We recall the basic factorization property of cartesian fibrations:

\begin{rem} \label{rem:cart_edge_factorization}
Let $F: \sX \to \sC$ be a cartesian fibration. Then given any edge $[e: z \to x]$ lying over an edge $[\alpha = F(e): d \to c]$, there exists a factorization $[z \xto{e'} y \xto{e''} x]$ of $e$ in which $e''$ is a $F$-cartesian edge lying over $\alpha$ and $e'$ lies in the fiber $\sX_{d}$. Moreover, this factorization is unique up to equivalence.
\end{rem}

\begin{rem} If $\delta: \sC \to \Delta^{\inj}$ is a discrete cartesian fibration, then every morphism in $\sC$ is necessarily a $\delta$-cartesian edge. Indeed, this follows immediately from Remark \ref{rem:cart_edge_factorization}.
\end{rem}

\begin{ntn} \label{ntn:simplexFace} Suppose $\delta: \sC \to \Delta^{\inj}$ is a discrete cartesian fibration (with $\sC$ a category, not necessarily a poset). Given an injective map of totally ordered sets $\alpha: [k] \to [n]$ with image $I \coloneq \alpha([k])$ in $[n]$, and $x \in P$ such that $\delta(x) = [n]$, let $\alpha^{\ast}(x)$ or $x_I$ denote the source object of the $\delta$-cartesian edge $\alpha^{\ast}(x) = x_I \to x$ covering $\alpha$, which is prescribed by the lifting property of $\delta$.
\end{ntn}

\begin{dfn} \label{dfn:deltaStructure} Let $P$ be a poset. A \emph{$\delta$-structure} on $P$ is the data of a functor $\delta: P \to \Delta^{\inj}$ that is a discrete cartesian fibration, along with the additional injectivity hypothesis:
\begin{itemize} \item[($\ast$)] If $x \in P$ with $\delta(x) = [n]$ and $\alpha, \beta: [k] \to [n]$ are two distinct maps in $\Delta^{\inj}$, then $\alpha^{\ast}(x) \neq \beta^{\ast}(x)$.
\end{itemize}

 $P$ is said to be \emph{$\delta$-admissible} if it admits a $\delta$-structure.
\end{dfn}

Let us pause to unwind Definition \ref{dfn:deltaStructure}. Let $\delta: P \to \Delta^{\inj}$ be a functor whose fibers are sets. Then the condition that $\delta$ be a cartesian fibration amounts to the following: for every $x \in P$ such that $\delta(x) = [n]$ and every injective map of totally ordered sets $\alpha: [k] \to [n]$, there exists a unique $\alpha^{\ast}(x) \leq x$ covering $\alpha$ such that for every $y \leq x$ covering $\gamma: [i] \to [n]$ with $\gamma([i]) \subset \alpha([k])$, we have that $y \leq \alpha^{\ast}(x)$.

We can also understand Definition \ref{dfn:deltaStructure} in terms of the equivalence between functors $X: (\Delta^{\inj})^{\op} \to \Set$ (i.e., semisimplicial sets) and discrete cartesian fibrations $\delta: \sC \to \Delta^{\inj}$ implemented by the Grothendieck construction. In this context, that equivalence is explicitly given as follows:
\begin{enumerate}
	\item Given a semisimplicial set $X$, define $\sC$ to be the \emph{category of simplices} of $X$. The objects of $\sC$ are given by $x \in X([n])$, while the morphisms $y \to x$, $x \in X([n])$ and $y \in X([k])$, are in bijective correspondence with $\alpha: [k] \to [n]$ such that $X(\alpha)(x) = y$. $\sC$ admits an obvious functor to $\Delta^{\inj}$ which is seen to be a discrete cartesian fibration.
	\item Conversely, given a discrete cartesian fibration $\sC \to \Delta^{\inj}$, define a semisimplicial set $X$ objectwise by $X([n]) = C_{[n]}$, and for every injective map of totally ordered sets $\alpha: [k] \to [n]$, define $X(\alpha): X([n]) \to X([k])$ by $X(\alpha)(x) = \alpha^{\ast}(x)$. 
\end{enumerate}

Under this correspondence, we say that the functor $X$ \emph{classifies} the cartesian fibration $\delta$.

Given $X$, its category of simplices is generally not a poset. Rather, we can refine the Grothendieck correspondence to one between $\delta$-structures $\delta: P \to \Delta^{\inj}$ and semisimplicial sets $X$ which satisfy the following additional condition:
\begin{itemize}
	\item[($\ast$)] For every $x \in X([n])$, the corresponding map $x: \Delta^n \to X$ is objectwise injective.\footnote{Here, $\Delta^n$ denotes the semisimplicial set which is the image of $[n]$ under the Yoneda embedding $\into{\Delta^{\inj}}{\Fun((\Delta^{\inj})^{\op}, \Set)}$, as is standard. Then maps $\Delta^n \to X$ are in bijection with elements of $X([n])$.}
\end{itemize}

\begin{exm} \label{exm:correspondence_delta_str} Let $K$ be an abstract simplicial complex with vertex set $V$ and let $P$ be its poset of simplices. Choosing an ordering of $V$, we can regard $K$ as a semisimplicial set and thereby produce a $\delta$-structure on $P$ by identifying $P$ with the category of simplices of $K$. Therefore, $P$ is $\delta$-admissible. In fact, we do not need to choose a global ordering of $V$ to see this; it is clear that the choice of $\delta$-structure on $P$ amounts to a compatible local ordering of the simplices of $K$. Conversely, if $P$ is $\delta$-admissible, then we may define an abstract simplicial complex $K$ with $P$ as its poset of simplices.
\end{exm}

\begin{exm} \label{exm:ssetsToPosets} Let $K$ be a simplicial set. Then the procedure of \cite[Var.~4.2.3.15]{HTT} produces a cofinal map $\phi: N(P) \to K$ that depends functorially on $K$, such that $P$ is a $\delta$-admissible poset.\footnote{The idea of Lurie's construction is to first consider the category of simplices $\Delta_{/K}$ together with the \emph{last vertex} map $\Delta_{/K} \to K$ that sends an $n$-simplex $\sigma \in K_n$ to $\sigma(n)$; this is cofinal by \cite[Prop.~4.2.3.14]{HTT}. One then modifies $\Delta_{/K}$ in a few clever ways to obtain a poset $P$ along with a cofinal $N(P) \to \Delta_{/K}$; the remaining problem is essentially to get around the presence of degenerate simplices in $\Delta_{/K}$. More precisely, $N(P)$ is the category of \emph{nondegenerate} simplices in the category $\Delta_{/(\Delta_{/K})} \times_{\Delta} \Delta^{\inj}$, where $\Delta_{/(\Delta_{/K})}$ denotes the category of simplices in $\Delta_{/K}$. By definition, $P$ then comes equipped with a map $\delta: P \to \Delta^{\inj}$ that constitutes a $\delta$-structure.} Furthermore, if $K$ is finite, then the procedure of \cite[Var.~4.2.3.16]{HTT} produces a cofinal map $\phi: N(P) \to K$ with $P$ a finite $\delta$-admissible poset; however, this construction is not functorial in $K$.

Recalling that if a map is cofinal then it is a weak homotopy equivalence (\cite[Prop.~4.1.1.3(3)]{HTT}), we see that the weak homotopy types of $\delta$-admissible posets encompass all spaces, with matching finiteness conditions.
\end{exm}

\begin{exm} In Example \ref{exm:ssetsToPosets}, if $K = N(Q)$ is itself the nerve of a poset, then the poset $P$ can be taken to be the subdivision $\sd(Q)$, defined to be the poset whose objects are $n$-chains $\sigma = [x_0 < \ldots < x_n]$ in $Q$ (i.e., non-degenerate simplices of $N(Q)$) with $\sigma \leq \tau$ if and only if $\sigma$ is a subchain of $\tau$. Then we have the functor $\delta: \sd(Q) \to \Delta^{\inj}$ given by the pullback of the structure map $\Delta_{/Q} \to \Delta$ of the category of simplices\footnote{$\Delta_{/Q}$ is the category of simplices of $Q$ regarded as a simplicial set $N(Q)$. $\sd(Q)$ is then the category of simplices of $Q$ regarded as a semisimplicial set, forgetting the degeneracies.} of $Q$, which is a discrete cartesian fibration, so $\sd(Q)$ is $\delta$-admissible. Define the ``last vertex'' map $\phi: \sd(Q) \to Q$ by $\phi([x_0 < \ldots < x_n]) = x_n$. Then $\phi$ is cofinal by the first argument of \cite[Var.~4.2.3.16]{HTT} (which doesn't use finiteness); the point is that we have a factorization $\sd(Q) \subset \Delta_{/Q} \to N(Q)$ of $\phi$ with the first map right adjoint to a retraction $\Delta_{/Q} \to \sd(Q)$ and the second map the cofinal map of \cite[Prop.~4.2.3.14]{HTT}. Moreover, the subdivision $\sd(-)$ is manifestly functorial in its argument.
\end{exm}

We also have the notion of dimension for posets and objects of posets.

\begin{dfn} \label{dfn:dimension} The \emph{dimension} of an object $x \in P$ is defined to be
\[ \dim(x) = \max\{n \: | \: \text{there exists a chain } x_0 < x_1 < ... < x_{n-1} < x_n = x \text{ in } P \}. \]
We will sometimes disambiguate $\dim(x)$ as $\dim_P(x)$ and refer to the \emph{$P$-dimension} of $x$.

The \emph{dimension} of $P$ is defined to be $\dim(P) = \max\{ \dim(x) \: | \: x \in P \}$ for $P$ nonempty (and equals $-1$ if $P$ is empty). $P$ is said to be finite-dimensional if $\dim(P) < \infty$.
\end{dfn}

Note that for any $\delta$-structure on $P$, $\delta(x) = [n]$ if and only if $\dim(x) = n$, and for $P$ finite-dimensional $\dim(P)$ equals the maximum $n$ such that $[n]$ is in the image of $\delta$. Every $\delta$-structure on $P$ thus knows the dimension of objects in $P$. However, we advise the reader not to think of the $\delta$-structure itself as a dimension function on $P$ (since the notion of dimension in $P$ doesn't vary under change of $\delta$-structure). Rather, the $\delta$-structure records face assignment under the correspondence of Example \ref{exm:correspondence_delta_str}.

\begin{ntn} In view of the fact that the objects of a $\delta$-admissible poset $P$ are simplices in a simplicial complex, or alternatively a semisimplicial set, we will henceforth typically denote objects of $P$ by $\sigma, \tau$, etc. rather than $x, y$, etc. If $\sigma \leq \tau$, then we will on occasion call $\sigma$ a \emph{face} of $\tau$ and $\tau$ a \emph{coface} of $\sigma$. 
\end{ntn}

\subsection{Links}

In this subsection, suppose that $P$ is a $\delta$-admissible poset. Our main result (Theorem \ref{thm:shiftingLink}) relates the value of the space-valued sheaf $\LL_P^{\sS}$ at any $\sigma \in P$ to the geometric realizations of two other subposets of $P$: the link and the small link of $\sigma \in P$.

\begin{ntn} Given a subposet $Q \subset P$, let $\overline{Q} \subset P$ denote the minimal sieve in $P$ containing $Q$. Viewing $P$ as a topological space, $\overline{Q}$ is the closure of $Q$ in $P$.
\end{ntn}

\begin{dfn} \label{dfn:link} The \emph{link} of $\sigma \in P$ is the subposet $\overline{P_{\geq \sigma}} - P_{\geq \sigma}$. The \emph{small link} of $\sigma \in P$ is the subposet $P_{>\sigma}$.\footnote{In the literature, authors sometimes refer to this subposet as the link of $\sigma$. We do not know of standard terminology which distinguishes between these two notions of link.}
\end{dfn}

To orient the reader, it may help to note that the link is the combinatorial analogue of the concept of punctured tubular neighborhood from manifold theory.

\begin{nul} \label{basicLinkRelation} The relevance of the link for us is that the unreduced suspension (Recollection \ref{rec:unreducedSuspension}) of the link of $\sigma$ is homotopy equivalent to $\LL_P^{\sS}(\sigma)$. To see this, first note that by Theorem \ref{thm:MayerVietoris} applied to the sieve covering $\left\{ \overline{P_{\geq \sigma}}, P - P_{\geq \sigma} \right\}$ of $P$, we have a pushout square of posets
\[ \begin{tikzcd}[row sep=4ex, column sep=4ex, text height=1.5ex, text depth=0.25ex]
\overline{P_{\geq \sigma}} - P_{\geq \sigma} \ar{r} \ar{d} & \overline{P_{\geq \sigma}} \ar{d} \\
P - P_{\geq \sigma} \ar{r} & P.
\end{tikzcd} \]
Upon geometric realization and using that $\left| \overline{P_{\geq \sigma}} \right| \simeq \ast$ by Lemma \ref{lm:cofinality}(1) below, we obtain a pushout square of spaces
\[ \begin{tikzcd}[row sep=4ex, column sep=4ex, text height=1.5ex, text depth=0.25ex]
\left| \overline{P_{\geq \sigma}} - P_{\geq \sigma} \right| \ar{r} \ar{d} & \ast \ar{d} \\
\left| P - P_{\geq \sigma} \right| \ar{r} & \left| P \right|.
\end{tikzcd} \]
Taking cofibers, we see that $S^1(| \overline{P_{\geq \sigma}} - P_{\geq \sigma} |) \simeq \LL_P^{\sS}(\sigma)$.
\end{nul}

There is a more subtle relation between the small link of $\sigma$ and the link of $\sigma$ (Theorem \ref{thm:shiftingLink}), which will occupy our attention in the remainder of this subsection. To prepare for the proof of Theorem \ref{thm:shiftingLink}, we need to introduce some auxiliary subposets. Fix a $\delta$-structure on $P$ and $\sigma \in P$ with $\delta(\sigma) = [n]$. Recalling Notation \ref{ntn:simplexFace}, define for every subset $I \subset [n]$ the subposet
\[ A_{I} \coloneq \{ \tau \in \overline{P_{\geq \sigma}} - P_{\geq \sigma} | \: \sigma_{\{j\}} \nleq \tau \text{ for all } j \notin I \} \subset \overline{P_{\geq \sigma}} - P_{\geq \sigma}. \]

We will use the posets $A_I$ to interpolate between the small link $P_{> \sigma}$ and the link $\overline{P_{\geq \sigma}} - P_{\geq \sigma}$ (cf. Lemmas~\ref{lem:PosetSuspension} and \ref{lm:slidingAbovePoset}). First note that for a proper nonempty subset $I \subset [n]$, $\sigma_I \in A_I$. Let $B_I \subset A_I$ be the cosieve in $A_I$ generated by $\sigma_I$. 

\begin{lem} \label{lm:cofinality} \begin{enumerate} \item The inclusion $i: P_{\geq \sigma} \to \overline{P_{\geq \sigma}}$ is cofinal. In particular, $\overline{P_{\geq \sigma}}$ is weakly contractible.
\item For every proper nonempty $I \subset [n]$, the inclusion $i_I: B_I \to A_I$ is cofinal. In particular, $A_I$ is weakly contractible.
\end{enumerate}
\end{lem}
\begin{proof} We will prove all of these cofinality claims by constructing retractions which are left adjoint to the inclusions and then invoking Quillen's Theorem A (\cite[Thm.~4.1.3.1]{HTT})\footnote{The $\infty$-categorical version of ``Quillen's Theorem A'' is due to Andr\'{e} Joyal.}, which implies that right adjoints are cofinal. The weak contractibility claim then follows because cofinal maps are weak homotopy equivalences and $P_{\geq \sigma}$ and $B_I$ are weakly contractible as they have initial objects.

(1): Let $r: \overline{P_{\geq \sigma}} \to P_{\geq \sigma}$ be the map which sends $\tau$ to the minimal $\tau'$ such that $\tau \leq \tau'$ and $\sigma \leq \tau'$. Then $r \circ i = \id$ and $r$ is left adjoint to $i$: the minimality of $r(\tau)$ precisely means that $r(\tau) \leq \kappa$ for any $\kappa \geq \sigma$ if and only if $\tau \leq \kappa$.

(2): Let $r_I: A_I \to B_I$ be the map which sends $\tau$ to the minimal $\tau'$ such that $\tau \leq \tau'$ and $\sigma_I \leq \tau'$; this is well-defined by our assumption that $\tau$ is in the sieve generated by $P_{\geq \sigma}$. Then by the same argument as in (1), $r_I$ is left adjoint to $i_I$.
\end{proof}



Let $\cP([n])$ be the poset of subsets of $[n]$ with the partial order given by inclusion. Note that if $I \subset I'$, then $A_I \subset A_{I'}$. We can thus define a functor $F: \cP([n]) \to \Cat_{\infty}$ by $F(I) = A_I$.

\begin{lem} \label{lem:PosetSuspension} $F: \cP([n]) \to \Cat_{\infty}$ is a colimit diagram, i.e. the canonical map
\[ \varinjlim_{I \subsetneq [n]} A_I \to \overline{P_{\geq \sigma}} - P_{\geq \sigma} \]
is an equivalence.
\end{lem}
\begin{proof} The following two facts are immediate from the definitions:
\begin{enumerate}
	\item The $A_I$ are sieves inside $\overline{P_{\geq \sigma}} - P_{\geq \sigma}$.
	\item Let $X_j \subset [n]$ be the subset excluding $j$. Then $A_I = \bigcap_{j \notin I} A_{X_j}$.
\end{enumerate}
The claim now follows from Theorem \ref{thm:MayerVietoris} applied to the identity functor and the cover of $\overline{P_{\geq \sigma}} - P_{\geq \sigma}$ by the collection of sieves $\{ A_{X_j}: \: j \in [n] \}$.
\end{proof}

\begin{lem} \label{lm:slidingAbovePoset} We have an isomorphism of posets $P_{> \sigma} \cong A_{\emptyset}$.
\end{lem}
\begin{proof} Given a proper inclusion $\iota: [n] \subset [m]$, let $\iota^c: [m-n-1] \subset [m]$ denote the complementary inclusion to $\iota$. For any $\kappa \geq \sigma$, let $\iota(\kappa) = \delta(\sigma \to \kappa)$. Define a map of sets $f: P_{> \sigma} \to A_{\emptyset}$ by $f(\kappa) = (\iota^c)^\ast(\kappa)$. If $\kappa \leq \kappa'$ with $\kappa$, $\kappa'$ of dimension $m$, $m'$ respectively, then we have a factorization
\[ \begin{tikzcd}[row sep=4ex, column sep=4ex, text height=1.5ex, text depth=0.25ex]
\left[ m-n-1 \right] \ar{r}{\beta} \ar{d}{\iota(\kappa)^c} & \left[ m'-n-1 \right] \ar{d}{\iota(\kappa')^c} \\
\left[ m \right] \ar{r}{\alpha} & \left[ m' \right]
\end{tikzcd} \]
where $\alpha^\ast(\kappa') = \kappa$, so
\[ f(\kappa) = (\alpha \circ \iota(\kappa)^c)^\ast (\kappa') = (\iota(\kappa')^c \circ \beta)^\ast(\kappa') = \beta^\ast (f(\kappa')), \]
hence $f(\kappa) \leq f(\kappa')$ and $f$ is a map of posets.

In the reverse direction, define a map of posets $g: A_{\emptyset} \to P_{> \sigma}$ sending $\tau$ to the minimum $\tau'$ such that $\tau \leq \tau'$ and $\sigma \leq \tau'$. Then it is easy to check that $f$ and $g$ are inverse to each other.
\end{proof}
\begin{rec} \label{rec:unreducedSuspension} Let $X \in \sS$ be a space and let $n>0$ be an integer. The \emph{$n$th unreduced suspension} $S^n(X)$ of $X$ is defined to be the colimit of the functor $F: \cP([n])_{<[n]} \to \sS$, $F(\emptyset) = X$ and $F(I) = \ast$ for every nonempty proper subset $I \subset [n]$; the functor $F$ can be precisely defined as the right Kan extension of $X: \{\emptyset\} \to \sS$ along the inclusion $\{\emptyset\} \subset \cP([n])_{<[n]}$.

If $X = \emptyset$, then $S^n(X) \simeq S^{n-1}$ is the $(n-1)$-sphere.\footnote{Unfortunately, we have a clash of notation here regarding the symbol $S^n$.} If $X$ is nonempty, then any choice of basepoint $x \in X$ identifies $S^n(X)$ with the $n$th reduced suspension $\Sigma^n X $, because $\Sigma^n X$ is computed by the same diagram in pointed spaces $\sS_{\ast}$ and colimits in $\sS_{\ast}$ taken over connected diagrams are computed the same as in $\sS$.

We have that $S^n(X) \simeq (S^1 S^1 \ldots S^1)(X)$, with the single unreduced suspension $S^1(-)$ iterated $n$ times.
\end{rec}

\begin{thm} \label{thm:shiftingLink} Let $\sigma \in P$ with $n = \dim(\sigma)$. Then the geometric realization of the link $\left| \overline{P_{\geq \sigma}} - P_{\geq \sigma} \right|$ is the $n$th unreduced suspension of the geometric realization of the small link $\left| P_{> \sigma} \right|$. Consequently,
\[ S^{n+1} \left| P_{> \sigma} \right| \simeq S^1 S^{n} \left| P_{> \sigma} \right|  \xto{\sim} S^1(\left| \overline{P_{\geq \sigma}} - P_{\geq \sigma} \right|) \xto{\sim} \LL_P^{\sS}(\sigma). \]
\end{thm}
\begin{proof} The first statement holds by taking the geometric realization of the diagram in Lemma \ref{lem:PosetSuspension}, using the weak contractibility result of Lemma \ref{lm:cofinality}(2), and identifying $A_{\emptyset}$ with $P_{> \sigma}$ by Lemma \ref{lm:slidingAbovePoset}. The consequence then follows by combining the first statement with the analysis of \ref{basicLinkRelation}.
\end{proof}



\begin{vrn} \label{vrn:functorialShiftingLink} Fix $\sigma \in P$ of dimension $n$. For every $\tau \in P_{> \sigma}$ and subset $I \subset [n]$, let $A_I^{\tau} = A_I \cap \overline{P_{\geq \sigma} - P_{\geq \tau}}$. Define a functor
\[ H: \cP([n]) \times P_{> \sigma} \to \Cat_{\infty} \]
by $H(I,\tau) = A_I^{\tau}$, with the functoriality given by the inclusion of subposets. Then for each $\tau > \sigma$, applying Lemma \ref{lem:PosetSuspension} with $P$ replaced by $P - P_{\geq \tau}$ shows that $H|_{\cP([n]) \times \{ \tau \}}$ is a colimit diagram, so the adjoint $H': \cP([n]) \to \Fun(P_{> \sigma}, \Cat_{\infty})$ is a colimit diagram. Postcomposing with geometric realization and using weak contractibility of the $A_I^\tau$ for $I \subset [n]$ nonempty proper (by Lemma \ref{lm:cofinality}(2) with $P$ replaced by $P - P_{\geq \tau}$), we see that $|H'([n])|$ is the $n$th unreduced suspension of $|H'(\emptyset)|$, as functors $P_{> \sigma} \to \sS$. Finally, using the isomorphisms $A^\tau_{\emptyset} \cong P_{> \sigma} - P_{\geq \tau}$ for all $\tau > \sigma$, we conclude that the functor
\[ B: P_{> \sigma} \to \sS, \: \tau \mapsto \left| \overline{P_{\geq \sigma} - P_{\geq \tau}} - (P_{\geq \sigma} - P_{\geq \tau}) \right| \]
is the $n$th unreduced suspension of the functor
\[ b: P_{> \sigma} \to \sS, \: \tau \mapsto \left| P_{> \sigma} - P_{\geq \tau} \right|.  \]

Now let
\[ F: \cP([n]) \to \Cat_{\infty}, \: I \mapsto A_I \]
be as before and let $F': \cP([n]) \to \Fun(P_{> \sigma}, \Cat_{\infty})$ be the composite of $F$ with the diagonal (i.e., the constant diagram functor). Via the inclusions $A_I^{\tau} \subset A_I$, we obtain a natural transformation $\eta: H' \to F'$. Taking geometric realizations, we get that $|\eta([n])| \simeq S^n(|\eta(\emptyset)|)$.

Concretely, this amounts to the following observation. Let $C, c: P_{> \sigma} \to \sS$ denote the constant functors at the geometric realization of the link $\left| \overline{P_{\geq \sigma}} - P_{\geq \sigma} \right|$ and the geometric realization of the small link $\left| P_{> \sigma} \right|$, respectively. Then the natural transformations $\phi: b \to c$ and $\Phi: B \to C$ defined by the evident inclusions satisfy the relation $\Phi \simeq S^n(\phi)$.
\end{vrn}

Using Variant \ref{vrn:functorialShiftingLink}, we can extend the objectwise equivalence of Theorem \ref{thm:shiftingLink} to an equivalence of functors.

\begin{prp} \label{prp:shiftedOrientationSheaf} Let $\sigma \in P$ of dimension $n$. Then we have an equivalence of functors 
\[ \Sigma^{n+1} \LL^{\sS}_{P_{>\sigma}} \simeq (\LL^{\sS}_P)|_{P_{>\sigma}} :P_{>\sigma} \to \sS_{\ast}. \]
\end{prp}
\begin{proof} Given $\tau > \sigma$, the sequence of inclusions
\[ P-P_{\geq \sigma} \to P-P_{\geq \tau} \to P \]
yields the cofiber sequence of pointed spaces
\[ \LL_{P - P_{\geq \tau}}^{\sS}(\sigma) \to \LL_P^{\sS}(\sigma) \to \LL_P^{\sS}(\tau) .\]
Let $F = \LL^{\sS}_{P-P_{\geq (-)}}(\sigma) : P_{> \sigma} \to \sS_{\ast}$ and $C(\sigma): P_{> \sigma} \to \sS_{\ast}$ be the constant functor at $\LL_P^{\sS}(\sigma)$. Then it follows that we have a cofiber sequence of functors $P_{>\sigma} \to \sS_\ast$
\[ F \to C(\sigma) \to (\LL_P^{\sS})|_{P_{>\sigma}}. \]
Now let $f: P_{> \sigma} \to \sS$ be defined by $f(\tau) = |P_{>\sigma} - P_{\geq \tau}|$ with the functoriality given by inclusion of subposets, and let $c(\sigma): P_{> \sigma} \to \sS$ be the constant functor at $|P_{> \sigma}|$. For all $\tau>\sigma$, we have cofiber sequences
\[ |P_{> \sigma} - P_{\geq \tau}| \to |P_{> \sigma}| \to \LL_{P_{>\sigma}}^{\sS}(\tau) \]
and this promotes to a cofiber sequence of functors $P_{>\sigma} \to \sS_\ast$
\[ f \to c(\sigma) \to \LL^{\sS}_{P_{>\sigma}}. \]

It follows from Variant \ref{vrn:functorialShiftingLink} that the natural transformation $F \to C(\sigma)$ is the $(n+1)$-unreduced suspension of $f \to c(\sigma)$. Unreduced suspension commutes with cofibers and is equivalent to ordinary suspension on pointed spaces, completing the proof.
\end{proof}

\subsection{Cosemisimplicial resolutions}
A choice of $\delta$-structure permits us to obtain cosemisimplicial resolutions computing the limit (i.e., global sections) of functors $F: P \to \sC$.

\begin{lem} \label{lm:cosimplicialResolution} Let $P$ be a poset equipped with a $\delta$-structure $\delta: P \to \Delta^{\inj}$. Let $\sC$ be an $\infty$-category with limits and let $F: P \to \sC$ be a functor. Then we have the following formula for the limit of $F$
\[ \varprojlim_{P} F \xto{\sim} {\varprojlim_{[n] \in \Delta^{\inj}}} \left( \prod_{\sigma \in P, \dim(\sigma)=n} F(\sigma) \right). \]
\end{lem}
\begin{proof} Because $\delta$ is a cartesian fibration, the right Kan extension of $F$ along $\delta$ is computed as the fiberwise limit and is given by the functor
\[ F^{\Delta}_{\bullet}: \Delta^{\inj} \to \sC, \: [n] \mapsto \prod_{\sigma \in P, \dim(\sigma)=n} F(\sigma). \]
The right Kan extension of $F$ along the projection of $P$ to the point computes the limit $\varprojlim_{P} F$. We now deduce the claim from the transitivity of right Kan extensions.
\end{proof}

Furthermore, we can exploit a `self-similarity' property of the category $\Delta^{\inj}$ in the form of Lemma \ref{lm:cartesianFibration} to obtain cosemisimplicial resolutions computing the limit of $F$ taken over subposets $P_{> \sigma}$. To state this, we need to introduce a bit of notation. For every $n \geq 0$, let $(\Delta^{\inj})^{\circ}_{[n]/}$ be the full subcategory of the slice category $(\Delta^{\inj})_{[n]/}$ excluding the initial object $[n] = [n]$. Given an inclusion $i: [n] \to [m]$, let $i^c: [m-n-1] \to [m]$ denote the complementary inclusion. Define a functor
$$\gamma_n: (\Delta^{\inj})_{[n]/}^{\circ} \to \Delta^{\inj}$$
on objects by $\gamma_n([n] \to [m]) = [m-n-1]$ and on morphisms
\[ \begin{tikzcd}[row sep=3ex, column sep=3ex, text height=1.5ex, text depth=0.25ex]
& \left[ n \right] \ar{ld}[swap]{i} \ar{rd}{j} & \\
\left[ m \right] \ar{rr}[swap]{f} & & \left[ m' \right] 
\end{tikzcd} \]
to be the unique map $\gamma_n(f): [m-n-1] \to [m'-n-1]$ which makes the diagram
\[ \begin{tikzcd}[row sep=4ex, column sep=6ex, text height=1.5ex, text depth=0.25ex]
\left[ m \right] \ar{r}{f} & \left[ m' \right] \\
\left[ m-n-1 \right] \ar{r}{\gamma_n(f)} \ar{u}{i^c} & \left[ m'-n-1 \right] \ar{u}{j^c}
\end{tikzcd} \]
commute.

\begin{lem} \label{lm:cartesianFibration} The functor $\gamma_n: (\Delta^{\inj})_{[n]/}^{\circ} \to \Delta^{\inj}$ is a cartesian fibration.
\end{lem}
\begin{proof} Let $g: [k] \to [k']$ be a morphism in $\Delta^{\inj}$ and let $j: [n] \to [n+k'+1]$ be an object in $(\Delta^{\inj})_{[n]/}^{\circ}$ covering $[k']$. Let $f: [n+k+1] \to [n+k'+1]$ be the morphism whose image is the union of that of $j$ and $g$, and let $i:[n] \to [n+k+1]$ be the factorization of $j$ through $f$. Then $f$, as a morphism $i \to j$ in $(\Delta^{\inj})_{[n]/}^{\circ}$, covers $g$, and we claim it is a $\gamma_n$-cartesian edge. We need to check that for any $h: [n] \to [m]$, the commutative square of sets
\[ \begin{tikzcd}[row sep=4ex, column sep=4ex, text height=1.5ex, text depth=0.25ex]
\Hom_{(\Delta^{\inj})_{[n]/}^{\circ}}(h,i) \ar{r}{f_{\ast}} \ar{d}{\gamma_n} & \Hom_{(\Delta^{\inj})_{[n]/}^{\circ}}(h,j) \ar{d}{\gamma_n} \\
\Hom_{\Delta^{\inj}} \left( \left[ m-n-1 \right], \left[ k \right] \right) \ar{r}{g_{\ast}} & \Hom_{\Delta^{\inj}}\left( \left[ m-n-1 \right],\left[ k' \right] \right)
\end{tikzcd} \]
is a pullback square. This amounts to the observation that given a commutative diagram
\[ \begin{tikzcd}[row sep=4ex, column sep=4ex, text height=1.5ex, text depth=0.25ex]
& \left[ n \right] \ar{rd}{j} \ar{ld}[swap]{h} & \\
\left[ m \right] \ar{rr}{\alpha} & & \left[ n+k'+1 \right]  \\
\left[ m-n-1 \right] \ar{r} \ar{u}{h^c} & \left[ k \right] \ar{r}{g} & \left[ k' \right] \ar{u}{j^c}
\end{tikzcd} \]
the map $\alpha$ factors uniquely through $[n+k+1]$ under $i$.
\end{proof}

\begin{nul} We interpret Lemma \ref{lm:cartesianFibration} as recording a `self-similarity' property of $\Delta^{\inj}$ in light of the following observation. Let $\sigma \in P$ be an object of dimension $n$. Define a functor
$$\chi: P_{>\sigma} \to (\Delta^{\inj})^{\circ}_{[n]/}$$
by $\chi(\tau) = \delta(\sigma \to \tau)$ and likewise on morphisms. It is easily checked that $\chi$ inherits the property of being a cartesian fibration from $\delta$. Together with Lemma \ref{lm:cartesianFibration}, we then see that the composition $\gamma_n \circ \chi: P_{>\sigma} \to \Delta^{\inj}$ is a discrete cartesian fibration.

Now applying Lemma \ref{lm:cosimplicialResolution} to the restriction of $F$ to $P_{>\sigma}$ with $\gamma_n \circ \chi$ taken as the $\delta$-structure on $P_{>\sigma}$, we obtain the formula
\[ \varprojlim_{P_{>\sigma}} F|_{P_{>\sigma}} \simeq {\varprojlim_{[k] \in \Delta^{\inj}}} \left( \prod_{\tau > \sigma, \dim(\tau) = n+k+1} F(\tau) \right).  \]
\end{nul}

Let us also record here a consequence of the above discussion.
\begin{cor} \label{cor:LinkAdmissible} If $P$ is $\delta$-admissible, then for any $\sigma \in P$, the subposet $P_{>\sigma}$ is $\delta$-admissible.
\end{cor}

\begin{nul}[\textbf{Bousfield-Kan spectral sequence}] \label{BKsseq} Let $\sC$ be a stable $\infty$-category equipped with a t-structure, e.g., $\sC = \Sp$ or $\categ{D}(k)$. Then we have the Bousfield-Kan spectral sequence for a cosemisimplicial object\footnote{cf. \cite[Var.~1.2.4.9]{HA}, bearing in mind that one can convert between a discussion of semisimplicial and cosemisimplicial objects by taking opposites and using that the opposite of a stable $\infty$-category with $t$-structure is again stable and has the opposite $t$-structure.}, which for $X_{\bullet}: \Delta^{\inj} \to \sC$ reads as
\[ E^1_{p,q} = \pi_q(X_p) \Rightarrow \pi_{p+q}\left( \varprojlim X_{\bullet} \right) \]

with the $d^1$-differential $d^1_{p,q}: \pi_q(X_p) \to \pi_q(X_{p+1})$ defined as $\pi_q$ of the alternating sum of the coface maps. For $F: P \to \sC$, we therefore have a spectral sequence
\[ E^1_{p,q} = \prod_{\sigma, \dim(\sigma)=p} \pi_q(F(\sigma)) \Rightarrow \pi_{p+q} \left( \varprojlim_P F \right)   \]

and for any $\sigma \in P$ of dimension $n$, we have a spectral sequence

\[ E^1_{p,q} = \prod_{\tau > \sigma, \dim(\tau) = n+p+1} \pi_q(F(\tau)) \Rightarrow \pi_{p+q} \left( \varprojlim_{P_{>\sigma}} F|_{P_{>\sigma}} \right). \]

The convergence of the Bousfield-Kan spectral sequence is generally a delicate matter. However, we have strong convergence if $X_k = 0$ for $k$ sufficiently large, which holds in our situation with $X_{\bullet} = F^{\Delta}_{\bullet}$ if $P$ is finite-dimensional.
\end{nul}

\begin{nul}[\textbf{Poincar\'{e} duality}] \label{PoincareDuality} Suppose now that $P$ is finite and let $F = \LL^{\ZZ}_P: P \to \categ{D}(\ZZ)$. Suppose that $F$ is locally constant at $\ZZ[n]$, $n = \dim(P)$. Then all the terms in Bousfield-Kan spectral sequence for $F^{\Delta}_{\bullet}$ vanish except for $E^1_{p,n}$, $0 \leq p \leq n$, and we obtain a cochain complex $D^{\bullet} = E^1_{\bullet,n}$ of abelian groups
\[ \begin{tikzcd}[row sep=4ex, column sep=6ex, text height=1.5ex, text depth=0.5ex]
\bigoplus\limits_{\sigma, \dim(\sigma)=0} \ZZ \ar{r}{d^1_{n,0}} & \bigoplus\limits_{\sigma, \dim(\sigma)=1} \ZZ \ar{r}{d^1_{n,1}} & \ldots \ar{r}{d^1_{n,n-2}} & \bigoplus\limits_{\sigma, \dim(\sigma)=n-1} \ZZ \ar{r}{d^1_{n,n-1}} & \bigoplus\limits_{\sigma, \dim(\sigma)=n} \ZZ
\end{tikzcd} \]

whose cohomology is the $E^2$-page of the spectral sequence. Because the $E^1$-page is concentrated on a single row and the $d^n$ differentials read as $d^n_{p,q}: E^n_{p,q} \to E^n_{p+n,q+n-1}$, the spectral sequence must degenerate at $E^2$, and we obtain the isomorphisms
\begin{equation} \label{eq:isomorphism}
H^p(D) \cong \pi_{n-p} \left( \varprojlim \LL^{\ZZ}_P \right) \text{ for all } p \in \ZZ.
\end{equation}

This proves a non-orientable Poincar\'{e} duality theorem in our context (but see Remark \ref{rem:standardPoincareDuality}). To see this, let $K$ be the semisimplicial set determined by $(P,\delta)$, and let $C_{\bullet}$ be the chain complex corresponding to the free semisimplicial abelian group $\ZZ[K]$ under the Dold-Kan correspondence, so $C_p = \ZZ[K_p] = \bigoplus_{\sigma, \dim(\sigma) = p} \ZZ$ and the boundary homomorphism $\partial_p: C_p \to C_{p-1}$ is given as the alternating sum of the face maps. Then $C_{\bullet}$ is a chain level model for $H \ZZ \otimes \Sigma^\infty_+ |P|$, which by Proposition \ref{prp:ExtendedOrientationSheaf} is equivalent to $\varprojlim \LL^{\ZZ}_P$.

On the other hand, the explicit description of the differentials $d^1_{n,p}$ shows that, up to sign, they are given as the transposes of the $\partial_{p+1}$. Indeed, the local constancy assumption on $\LL^{\ZZ}_P$ ensures that for all $\sigma < \tau$, the maps $\pi_n \LL^{\ZZ}_P(\sigma) \cong \ZZ \to \pi_n \LL^{\ZZ}_P(\tau) \cong \ZZ$ are given by multiplication by $u_{\sigma<\tau} = \pm 1$ (here, we use the same isomorphisms defining $D^{\bullet}$). The claim then follows by observing that for $\sigma$ and $\tau$ of dimension $p$ resp. $p+1$, by definition the $(\tau,\sigma)$ entry of the map $d^1_{n,p}$ is $0$ if $\sigma \nless \tau$ and $(-1)^i u_{\sigma<\tau}$ if $\sigma < \tau$ and $d_i(\tau) = \sigma$. In sum, we thus see that the isomorphism \eqref{eq:isomorphism} involves homology groups on the right and \emph{twisted} cohomology groups on the left.

Upon taking $\FF_2$-coefficients, we may identify $d^1_{n,p}$ with $\partial_{p+1}^T$, thereby deducing the isomorphisms 
\[ H^p(C; \FF_2) \cong H_{n-p}(C; \FF_2). \]
If we can choose a system of generators of the $\pi_n \LL^{\ZZ}_P(\sigma)$ which are compatible in the sense that all of the signs $u_{\sigma<\tau}$ equal $1$, then we moreover have the integral isomorphisms
\[ H^p(C; \ZZ) \cong H_{n-p}(C; \ZZ). \]
We thus see that the choice of such a compatible system is the combinatorial analogue of orienting a manifold.
\end{nul}

\begin{nul}[\textbf{Orientability}] \label{orientablePoincareDuality} Suppose that $P$ is finite, of dimension $n$, and connected, and that the local $\FF_2$-homology sheaf $\LL^{\FF_2}_P$ is locally constant (necessarily at $\Sigma^{\infty} S^n \otimes H \FF_2 \simeq \Sigma^n H \FF_2$). Then by \ref{PoincareDuality}, $H_{n}(|P|; \FF_2) \cong H^0(|P|; \FF_2) \cong \FF_2$. Let $f: \Sigma^{n} H \FF_2 \to \Sigma^{\infty}_+|P| \otimes H \FF_2$ be a map corresponding to a generator $g$ of $H_n(|P|; \FF_2)$. Examining the spectral sequence, we see that $g$ maps to a generator of $\pi_n(\LL_P^{\FF_2}(\sigma)) \cong \FF_2$ for all $\sigma \in P$ with $\dim(\sigma)=0$, hence for all $\sigma \in P$ by local constancy of $\LL_P^{\FF_2}$ and connectedness of $P$. Thus, for all $\sigma \in P$ the composite map
\[ \Sigma^{n} H \FF_2 \xto{f} \Sigma^{\infty}_+|P| \otimes H \FF_2 \simeq \varprojlim_{\sigma \in P} \LL_P^{\FF_2} \to \LL_P^{\FF_2}(\sigma) \]
is an equivalence. Let $c[n]: P \to \categ{D}(\FF_2)$ denote the constant functor at $\Sigma^{n} H \FF_2$ and let $\theta: c[n] \to \LL_P^{\FF_2}$ be the natural transformation adjoint to $f$. Because equivalences in functor categories are checked objectwise, it follows that $\theta$ is an equivalence. In other words, the locally constant sheaf $\LL^{\FF_2}_P$ is actually constant.

Now suppose that $H_n(|P|; \ZZ) \cong \ZZ$. Let $f': \Sigma^{n} H \ZZ \to \Sigma^{\infty}_+|P| \otimes H \ZZ$ be a map corresponding to a generator of $H_n(|P|; \ZZ)$. By the same reasoning as above, the natural transformation adjoint to $f'$ is an equivalence, showing that $\LL^{\ZZ}_P$ is a constant sheaf. Constancy of $\LL^{\ZZ}_P$ then permits us to choose a compatible system of generators of $\pi_n \LL^{\ZZ}_P$ in the sense of \ref{PoincareDuality}, which yields Poincar\'{e} duality integrally as for orientable manifolds. Even better, constancy of $\LL^{\ZZ}_P$ implies that we have an equivalence in $\categ{D}(\ZZ)$
\[ \Sigma^{\infty}_+ |P| \otimes H \ZZ \simeq \varprojlim_{P} \LL_P^{\ZZ} \simeq F(\Sigma^{\infty}_+ |P|, \Sigma^{n} H \ZZ). \]
\end{nul}

\begin{rem} \label{rem:standardPoincareDuality}
Suppose that $P$ is finite and let $n = \dim(P)$; for instance, $P$ could be the poset of simplices of a finite triangulation of a possibly non-orientable closed connected $n$-manifold. Define the \emph{orientation sheaf} of $P$ by $\omega_P \coloneq \Sigma^{-n} \LL^{\ZZ}_P$. Recall that by definition, given a sheaf $\cF$ on $P$, its $k$th cohomology group is defined to be $\pi_{-k}$ of the global sections
$$ H^k(P; \cF) \coloneq \pi_{-k}(\varprojlim \cF).$$
In particular, since $\varprojlim \Sigma^{-n} \LL^{\ZZ}_P \simeq \Sigma^{-n} H \ZZ \otimes \Sigma^{\infty}_+ |P|$ (by Proposition \ref{prp:ExtendedOrientationSheaf} and using that desuspension commutes with limits), we have that $H^k(P; \omega_P) \cong H_{n-k}(P; \ZZ)$, which matches with the familiar non-orientable Poincar\'{e} duality isomorphism from manifold theory. In particular, note that the isomorphism \eqref{eq:isomorphism} of \ref{PoincareDuality} is not merely this statement.
\end{rem}





\section{Stratification} \label{sec:stratification}

In this section, let $P$ be a finite $\delta$-admissible poset. In other words, there exists an abstract finite simplicial complex $K$ such that $P$ is the poset of simplices of $K$ (cf. Example \ref{exm:correspondence_delta_str}).

\begin{dfn} \label{dfn:strat} Let $n = \dim(P)$. A map of posets $\pi: P \to [n]$ is a \emph{stratification} of $P$ if it possesses the following property:

\begin{itemize}
\item[($\ast$)] For every $0 \leq i \leq n$, regard $[i]$ as the subposet $\{0<1<...<i\}$ of $[n]$, and let
\[ P_{\pi \leq i} \coloneq [i] \times_{[n]} P, \quad P_{\pi = i} \coloneq \{i\} \times_{[n]} P. \]
Then the restriction of $\LL_{P_{\pi \leq i}}$ to $P_{\pi = i}$ is locally constant at $\Sigma^\infty S^i$.\footnote{Note that this condition is vacuous if $P_{\pi = i} = \emptyset$.}
\end{itemize}


We say that a stratification $\pi$ is \emph{canonical} if for every $0 \leq i \leq n$, $P_{\pi = i} \subset P_{\pi \leq i}$ is maximal among all cosieves $U \subset P_{\pi \leq i}$ such that $(\LL_{P_{\pi \leq i}})|_U$ is locally constant at $\Sigma^{\infty} S^i$.

If the stratification $\pi$ is canonical, we call $P_{\pi=i}$ the \emph{$i$-strata} and $P_{\pi=n}$ the \emph{generic strata}. We also call a connected component of $P_{\pi=i}$ (resp. $P_{\pi=n}$) an \emph{$i$-stratum} (resp. \emph{generic stratum}).
\end{dfn}

\begin{rem} It is clear from the definition that canonical stratifications exist and are unique, so we will speak of \emph{the} canonical stratification of $P$.
\end{rem}

\begin{exm} The dimension map $\dim: P \to [\dim(P)]$ is a stratification of $P$, but it is generally not canonical (cf. the examples in the introduction).
\end{exm}

\begin{exm}
Let $P$ be the poset of simplices of a finite triangulation of a closed $n$-manifold. Define the map $\pi: P \to [n]$ to be constant at the value $\{ n \}$. Then $\pi$ is the canonical stratification of $P$ and $P$ equals its own generic strata.
\end{exm}

\begin{exm}
To generalize the example of the $2$-disk from the introduction, let $P$ be the poset of simplices of a finite triangulation of a compact $n$-manifold $M$ with nonempty boundary $\partial M$, where the triangulation is chosen to be compatible with the inclusion $\partial M \subset M$. Let $\partial P \subset P$ be the subposet on those simplices belonging to $\partial M$. Define the map $\pi: P \to [n]$ by sending $\partial P$ to $\{ n-1 \}$ and all other simplices to $\{ n \}$. Then $\pi$ is the canonical stratification of $P$, and the generic strata is given by those simplices in the interior of $M$.
\end{exm}

\begin{exm}
Consider the closed $3$-cube $I^3 = [0,1]^{\times 3}$, let $K$ be a finite triangulation of $I^3$ whose vertices include the 8 endpoint vertices, and let $P$ be the poset of simplices of $K$. Then its canonical stratification $\pi: P \to [3]$ is a surjective map whose fiber over $\{ 0 \}$ consists exactly of those 8 vertices.
\end{exm}

\begin{rem} \label{rem:stableVsUnstable} The notion of stratification in Definition \ref{dfn:strat} is \emph{stable} in the sense of being defined with reference to the $\Sp$-valued sheaf $\LL_P$. We could have considered the corresponding unstable notion with $\LL_P^{\sS}$ in place of $\LL_P$, but this would not be amenable to practical computation.
\end{rem}

\begin{rem} [Localization] \label{rem:localization} Suppose that $A$ is any poset. Let $\Str_{A}$ denote the full subcategory of $(\Cat_{\infty})_{/A}$ consisting of those functors $\pi: \sC \to A$ which are \emph{conservative}, in the sense that for any morphism $e: x \to y$ in $\sC$, if $\pi(e)$ is an equivalence then $e$ is an equivalence. Note that the fibers of a conservative functor $\pi: \sC \to A$ are spaces. As discussed in \cite{HaineStrat,DouteauThesis,Douteau2020}, the $\infty$-category $\Str_A$ should be thought of as the $\infty$-category of $A$-\emph{stratified spaces}. By \cite[Constr.~2.2.3]{Exodromy}, the fully faithful inclusion $\iota_A: \Str_{A} \to (\Cat_{\infty})_{/A}$ admits a left adjoint $\Ex^{\infty}_A$ that specializes to the adjunction
\[ \adjunct{|-|}{\Cat_{\infty}}{\sS}{\iota} \]
in the case that $A = [0]$ is the terminal poset.

Now suppose that $f: A \to B$ is a map of posets. Recall that we have the adjunction
\[ \adjunct{f_!}{(\Cat_{\infty})_{/A}}{(\Cat_{\infty})_{/B}}{f^{\ast}} \]
where $f_!(\pi: \sC \to A) = (f\circ\pi: \sC \to B)$ and $f^{\ast}(\sC \to B) = (\sC \times_B A \to A)$. Because the pullback of a conservative functor is again conservative, $f^{\ast}$ restricts to a functor $f^{\ast}: \Str_B \to \Str_A$. We then have the induced adjunction
\[ \adjunct{\Ex^{\infty}_B f_! i_A}{\Str_A}{\Str_B}{f^{\ast}} \]
by the usual observation regarding mapping spaces
\begin{align*}
\Map_{\Str_B}(\Ex^{\infty}_B f_! i_A (\sC \to A), (\sD \to B)) & \simeq \Map_{(\Cat_{\infty})_{/B}}(f_! i_A (\sC \to A), (\sD \to B)) \\
& \simeq \Map_{(\Cat_{\infty})_{/A}} (i_A (\sC \to A), f^{\ast}(\sD \to A)) \\
& \simeq \Map_{\Str_A}((\sC \to A), f^{\ast}(\sD \to A)).
\end{align*}

In particular, letting $f: A \to [0]$ be the map from $A$ to the terminal poset, we get that
$$|\sC| \xto{\sim} |\Ex^{\infty}_A(\sC \to A)|,$$
where $|\Ex^{\infty}_A(\sC \to A)|$ denotes the geometric realization of the domain $\infty$-category and the map is induced by the unit of the adjunction $\Ex^{\infty}_A \dashv \iota_A$.

Furthermore, the explicit construction of $\Ex_A^{\infty}$ in \cite[Constr.~2.2.3]{Exodromy} shows that the diagram
\[ \begin{tikzcd}[row sep=4ex, column sep=4ex, text height=1.5ex, text depth=0.5ex]
(\Cat_{\infty})_{/B} \ar{r}{\Ex^{\infty}_B} \ar{d}{f^{\ast}} & \Str_B \ar{d}{f^{\ast}} \\
(\Cat_{\infty})_{/A} \ar{r}{\Ex^{\infty}_A} & \Str_A
\end{tikzcd} \]
commutes. In particular, letting $f: \{x\} \to A$ be the inclusion of an object, we get an identification of the fiber $(\Ex^{\infty}_A(\sC \to A))_x \simeq |\sC_x|$.
\end{rem}

\begin{dfn} Let $n = \dim(P)$ and $\pi: P \to [n]$ be the canonical stratification of $P$. The \emph{canonical stratified homotopy type} of $P$ is the functor $\Pi^{\can} \coloneq \Ex^{\infty}_{[n]}(\pi): \sP^{\can} \to [n]$.
\end{dfn}


The primary goal of this section is to give a computationally tractable algorithm (Algorithm \ref{algorithm}) for determining the canonical stratification of $P$. The main engine behind this algorithm is Proposition \ref{prp:extendingLocalConstancy}, which concerns how to, given local constancy of $\LL_P$ on $P_{>\sigma}$, determine whether or not $(\LL_P)|_{P_{\geq \sigma}}$ is locally constant. We prepare for the proof of that proposition with the following lemma.

\begin{lem} \label{lm:basicFiberSequence} Let $\sigma \in P$ of dimension $d$. Then we have a fiber sequence of spectra
\[ \Sigma^\infty S^d \to \LL_P(\sigma) \to \varprojlim_{\tau \in P_{>\sigma}} \LL_P(\tau). \]
\end{lem}
\begin{proof} By Proposition \ref{prp:ExtendedOrientationSheaf}, we have an equivalence
\[ \Sigma^\infty |P|/|P-P_{> \sigma}| \simeq \varprojlim_{\tau \in P_{>\sigma}} \LL_P(\tau) \]
under which the canonical map
\[ f: \LL_P(\sigma) \to \varprojlim_{\tau \in P_{>\sigma}} \LL_P(\tau) \]
is identified with the map
\[ \Sigma^\infty |P|/|P-P_{\geq \sigma}| \to \Sigma^\infty |P|/|P-P_{> \sigma}| \]
induced by the inclusions $P-P_{\geq \sigma} \subset P-P_{> \sigma} \subset P$. $f$ then fits into a fiber sequence
\[ \Sigma^\infty |P-P_{> \sigma}|/|P-P_{\geq \sigma}| \to \LL_P(\sigma) \xto{f} \Sigma^\infty |P|/|P-P_{> \sigma}|. \]
The two sieves $P-P_{\geq \sigma}$ and $P_{\leq \sigma}$ cover $P-P_{> \sigma}$, so by Theorem \ref{thm:MayerVietoris} we have a pushout square of posets
\[ \begin{tikzcd}[row sep=4ex, column sep=4ex, text height=1.5ex, text depth=0.25ex]
P_{< \sigma} = P_{\leq \sigma} \cap (P-P_{\geq \sigma}) \ar{r} \ar{d} & P_{\leq \sigma} \ar{d} \\
P-P_{\geq \sigma} \ar{r} & P-P_{> \sigma}
\end{tikzcd} \]
and hence an equivalence of pointed spaces $|P_{\leq \sigma}|/|P_{< \sigma}| \xto{\sim} |P-P_{> \sigma}|/|P-P_{\geq \sigma}|$. Finally, note that $|P_{\leq \sigma}|/|P_{< \sigma}| \simeq S^d$ since $P_{< \sigma}$, resp. $P_{\leq \sigma}$ is the category of simplices of $\partial \Delta^{d}$, resp. $\Delta^d$.
\end{proof}

\begin{prp} \label{prp:extendingLocalConstancy} Let $\sigma \in P$ of dimension $d$ with a successor of top dimension $n = \dim(P)$, $n>d$. Suppose $(\LL_P)|_{P_{>\sigma}}$ is locally constant. Then $(\LL_P)|_{P_{\geq \sigma}}$ is locally constant if and only if $\Sigma^\infty_+ |P_{>\sigma}| \simeq \Sigma^\infty_+ S^{n-1-d}$.
\end{prp}
To understand the statement of Proposition \ref{prp:extendingLocalConstancy}, the reader may want to keep in mind the following picture: suppose $P$ is the poset of simplices of a finite triangulation of the $2$-disk $D^2$ as in the introduction, and $x$ is a vertex. Then $|P_{>x}|$ is homotopy equivalent to a circle if and only if $x$ is an interior vertex.

\begin{proof} Note that if $\kappa$ is a simplex of top dimension $n$, then $\LL_P(\kappa) \simeq \Sigma^\infty S^n$. Therefore, the value of $\LL_P$ on any $\tau>\sigma$ is necessarily $\Sigma^\infty S^n$ in view of the local constancy hypothesis on $(\LL_P)|_{P_{>\sigma}}$. 

First suppose $(\LL_P)|_{P_{\geq \sigma}}$ is locally constant. Then $\LL_P(\sigma) \simeq \Sigma^\infty S^n$. On the other hand, by Theorem \ref{thm:shiftingLink} we have that $\LL_P(\sigma) \simeq \Sigma^\infty S^d |P_{> \sigma}|$. Choosing any basepoint of $|P_{> \sigma}|$, we deduce that $\Sigma^\infty |P_{> \sigma}| \simeq \Sigma^\infty S^{n-1-d}$.

Conversely, suppose that $\Sigma^\infty_+ |P_{>\sigma}| \simeq \Sigma^\infty_+ S^{n-1-d}$. Then by Theorem \ref{thm:shiftingLink}, we get that $\LL_P(\sigma) \simeq \Sigma^\infty S^n$. Moreover, the map $\Sigma^\infty S^d \to \LL_P(\sigma)$ of Lemma \ref{lm:basicFiberSequence} is zero because $\pi^s_d(S^n) = 0$ for $d<n$. Shifting the fiber sequence of Lemma \ref{lm:basicFiberSequence} over by one to the right, we obtain a split fiber sequence
\[ \LL_P(\sigma) \simeq \Sigma^\infty S^n \to \varprojlim (\LL_P)|_{P_{> \sigma}} \simeq \Sigma^\infty S^n \oplus \Sigma^\infty S^{d+1} \to \Sigma^\infty S^{d+1}.  \]


If $d=n-1$, then $P_{>\sigma}$ is discrete and must have exactly two objects $\tau_0, \tau_1$, and the restriction of either map $\varprojlim (\LL_P)|_{P_{> \sigma}} \to \LL_P(\tau_i)$, $i=0,1$ to the summand $\Sigma^\infty S^n$ is necessarily an equivalence. Thus the maps $\LL_P(\sigma) \to \LL_P(\tau_i)$, $i=0,1$ are equivalences, so $(\LL_P)|_{P \geq \sigma}$ is locally constant.

Now suppose $d<n-1$. We claim that for any $\tau > \sigma$, the restriction of the canonical map
$$\varprojlim (\LL_P)|_{P_{> \sigma}} \to \LL_P(\tau) \simeq \Sigma^\infty S^n$$
to the summand $\Sigma^\infty S^n$ is an equivalence, i.e., a degree $\pm 1$ map. For this, it suffices to show that for any $\tau>\sigma$ with $\dim(\tau) = d+1$, the map
\[ \pi_n \varprojlim (\LL^{\ZZ}_P)|_{P_{> \sigma}} \cong \ZZ \to \pi_n \LL_P^{\ZZ}(\tau) \cong \ZZ \] 
is an isomorphism (note here that the finite limit commutes with base change to $\categ{D}(\ZZ)$). In fact, we will show that the map
\[ \alpha: \pi_n \varprojlim (\LL^{\ZZ}_P)|_{P_{> \sigma}} \cong \ZZ \to \bigoplus_{\tau > \sigma, \dim(\tau) = d+1} \pi_n \LL_P^{\ZZ}(\tau) \cong \bigoplus_{\tau > \sigma, \dim(\tau) = d+1} \ZZ \] 
is injective and sends $1$ to a vector of $\pm 1$s. For this, we use the Bousfield-Kan spectral sequence for a cosemisimplicial object set up in \ref{BKsseq}, applied to $(\LL^{\ZZ}_P)|_{P_{>\sigma}}$. We have
\[ E^1_{p,q} = \pi_q \left( \bigoplus_{\tau > \sigma, \dim(\tau) = d+q+1} \LL_P^{\ZZ}(\tau) \right) \Rightarrow \pi_{p+q} \varprojlim (\LL^{\ZZ}_P)|_{P_{> \sigma}}. \]
Because $(\LL_P)|_{P_{>\sigma}}$ is locally constant at $\Sigma^{\infty} S^n$, the $E^1$ page is concentrated on the $n$th row and is given by the cochain complex
\[ \begin{tikzcd}[row sep=4ex, column sep=4ex, text height=1.5ex, text depth=0.25ex]
\bigoplus_{\tau > \sigma, \dim(\tau) = d+1} \ZZ \ar{r}{\partial^0} & \bigoplus_{\tau > \sigma, \dim(\tau) = d+2} \ZZ \ar{r}{\partial^1} & ... \ar{r}{\partial^{n-d-1}} & \bigoplus_{\tau > \sigma, \dim(\tau) = n} \ZZ
\end{tikzcd} \]

Convergence of the spectral sequence then implies that map $\alpha$ fits into the short exact sequence
\[ \begin{tikzcd}[row sep=4ex, column sep=4ex, text height=1.5ex, text depth=0.25ex]
0 \ar{r} & \ZZ \ar{r}{\alpha} & \bigoplus_{\tau > \sigma, \dim(\tau) = d+1} \ZZ \ar{r}{\partial^0} & \im(\partial^0) \ar{r} & 0.
\end{tikzcd} \]

In particular, the rank of $\ker(\partial^0)$ equals one. As explained in \ref{PoincareDuality}, the boundary homomorphism $\partial^0$ is given up to sign as the transpose of the incidence matrix
\[ I: \bigoplus_{\tau> \sigma, \dim(\tau)=d+2} \ZZ \to \bigoplus_{\tau>\sigma, \dim(\tau)=d+1} \ZZ. \]
In particular, with $\FF_2$ coefficients $\partial^0 = I^T$. Note that we cannot have more than a single connected component of $P_{>\sigma}$ because that would increase the rank of $\ker(I^T)$ beyond one. Furthermore, for a vector $(k_\tau)$ to lie in the kernel of $\partial^0$, we must have for every $\kappa>\sigma$ of $\dim(\kappa)=d+2$ that $k_{\tau} = \pm k_{\tau'}$ for the two faces $\tau, \tau'$ of $\kappa$ above $\sigma$. Connectedness then ensures that $k_{\tau} = \pm k_{\tau'}$ for all $\tau, \tau'$, so a generator of $\ker(\partial^0)$ is given by a vector of $\pm 1$s, as desired.

Finally, for any $\tau > \sigma$, factorization of the map $\LL_P(\sigma) \to \LL_P(\tau)$ as
\[ \LL_P(\sigma) \simeq \Sigma^\infty S^n \to \varprojlim (\LL_P)|_{P_{> \sigma}} \simeq \Sigma^\infty S^n \oplus \Sigma^\infty S^{d+1} \to \LL_P(\tau) \simeq \Sigma^\infty S^n \]
shows that $\LL_P(\sigma) \to \LL_P(\tau)$ is an equivalence, completing the proof.
\end{proof}

\begin{rem} \label{rem:Hurewicz} By the stable Hurewicz theorem, $\Sigma^\infty_+ |P_{>\sigma}| \simeq \Sigma^\infty_+ S^{n-1-d}$ if and only if $H_{\ast}(|P_{>\sigma}|; \ZZ) \cong H_{\ast}(S^{n-1-d}; \ZZ)$. Therefore, in the definition of canonical stratification, we may as well replace $\LL_P$ by $\LL_P^{\ZZ}$. This has the practical effect of making the canonical stratification amenable to machine computation.
\end{rem}

\begin{nul}[\textbf{Algorithm for canonical stratification}] \label{algorithm} Proposition \ref{prp:extendingLocalConstancy} gives an iterative procedure for constructing the canonical stratification $\pi: P \to [n]$ of $P$, $n = \dim(P)$. Initialize a subposet $G \subset P$ to consist of all $\sigma \in P$ of dimension $n$. Iteratively add objects $\sigma \in P$ to $G$ according to the following rule:

\begin{itemize}
\item[$\bullet$] Suppose given $\sigma \in P$ with $\dim(\sigma) = d$ such that for all $\tau > \sigma$, we have $\tau \in G$. Then if $H_{\ast}(|P_{>\sigma}|; \ZZ) \cong H_{\ast}(S^{n-1-d}; \ZZ)$, add $\sigma$ to $G$.
\end{itemize}

This process terminates and defines a cosieve $G \subset P$, which we call the \emph{generic strata} of $P$. In view of Proposition \ref{prp:extendingLocalConstancy} and Remark \ref{rem:Hurewicz}, $G$ is the maximal cosieve in $P$ such that the restriction of $\LL_P$ to $G$ is locally constant at $\Sigma^{\infty} S^n$.

Next let $P^1 = P - G$. Then $\dim(P^1) < \dim(P)$, and we may repeat this procedure with $P^1$ in place of $P$.\footnote{We use implicitly that any sieve of a $\delta$-admissible poset is again $\delta$-admissible, because the restriction of a discrete cartesian fibration $\delta: P \to \Delta^{\inj}$ to any sieve remains a discrete cartesian fibration.} We thereby determine the maximal cosieve $G^1 \subset P^1$ such that the restriction of $\LL_{P^1}$ to $G^1$ is locally constant at $\Sigma^{\infty} S^{\dim(P^1)}$. 

Continuing, we end up with a filtration of $P$ by sieves
\[ P = P^0 \supset P^1 \supset \ldots \supset P^k = \emptyset \]
such that $(\LL_{P^i})|_{P^i - P^{i-1}}$ is locally constant at $\Sigma^{\infty} S^{\dim(P^i)}$. If we then define $\pi: P \to [n]$ by $\pi(\sigma) = \min\{ \dim(P^i) \: | \: \sigma \in P^i \}$, then $\pi$ is the canonical stratification of $P$.
\end{nul}

\begin{rem} \label{rem:reduceComputationPoincareDuality} Suppose we are in the situation of Proposition \ref{prp:extendingLocalConstancy} and wish to check if $H_{\ast}(|P_{>\sigma}|; \ZZ) \cong H_{\ast}(S^{n-1-d}; \ZZ)$ for $\sigma$ of dimension $d$ in order to determine if $\sigma$ belongs to the generic strata. Then we may exploit the Poincar\'{e} duality results of \ref{PoincareDuality} and \ref{orientablePoincareDuality} together with the universal coefficients theorems (stated for $A$ an abelian group)
\[ 0 \to \Ext^1(H_{i-1}(|P_{>\sigma}|; \ZZ) , A) \to H^i(|P_{>\sigma}|;A) \to \Hom(H_i(|P_{>\sigma}|; \ZZ), A) \to 0 \]
\[ 0 \to H_i(|P_{>\sigma}|;\ZZ) \otimes A \to H_i(|P_{>\sigma}|;A) \to \Tor^1(H_{i-1}(|P_{>\sigma}|;\ZZ), A) \to 0   \]
to reduce the amount of needed computation. Suppose $d<n-1$, $n = \dim(P)$ (the case $d = n-1$ being only a check as to whether $P_{>\sigma}$ has exactly two elements). Then the check for membership of $\sigma$ in the generic strata proceeds as follows (where we terminate with a \emph{negative} response if at any point the computed quantity fails to be as indicated):
\begin{enumerate}
	\item[0.] By Corollary \ref{cor:LinkAdmissible} and applying the Dold-Kan correspondence, a chain complex $C_{\ast}$ for computing $H_{\ast}(|P_{>\sigma}| ; \ZZ)$ is given by letting $C_{i}$ be the free abelian group on $\tau \in P_{>\sigma}$ with $\dim(\tau) - d- 1 = i$ and defining the boundary homomorphisms by the chosen $\delta$-structure.
	
	\item[1.] First compute $H_0(|P_{>\sigma}|; \ZZ) = \ZZ$, i.e., show that $P_{>\sigma}$ has a single connected component.

	\noindent -- By Poincar\'{e} duality for $\FF_2$-coefficients, this shows that $H_{n-1-d}(|P_{>\sigma}|; \FF_2) = \FF_2$.

    \item[2\text{a}.] If $n-1-d = 1$, terminate.

    \noindent -- We are done because the universal coefficients theorem for homology shows that $H_1(|P_{>\sigma}|; \ZZ) = \ZZ$.
	
	\item[2\text{b}.] If $n-1-d = 2$, compute the Euler characteristic $\chi(P_{>\sigma}) = 2$, and terminate.

	\noindent -- Let $r[p] = \rank H_1(|P_{>\sigma}|; \FF_p)$ and $s[p] = \rank H_2(|P_{>\sigma}|; \FF_p)$ for any prime $p$. Then $\chi(P_{>\sigma}) = 1-r[2]+s[2] = 2-r[2] = 2$ shows that $r[2] = 0$, hence as in step (2c) below we deduce that $H_2(|P_{>\sigma}|; \ZZ) = \ZZ$. Then $\chi(P_{>\sigma}) = 2 - r[p] = 2$ shows that $r[p] = 0$ for all primes $p$, hence $H_1(|P_{>\sigma}|; \ZZ) = 0$, and we are done.
	
	\item[2\text{c}.] If $n-1-d>2$, compute $H_1(|P_{>\sigma}|; \ZZ) = 0$.

	\noindent -- This shows $H_{n-d-2}(|P_{>\sigma}|; \FF_2) \cong H_1(|P_{>\sigma}|; \FF_2) = 0$, so by the universal coefficients theorem for homology we deduce that $H_{n-1-d}(|P_{>\sigma}|; \ZZ) = \ZZ$. By \ref{orientablePoincareDuality}, we now have Poincar\'{e} duality integrally.
	
	\item[2\text{c}--i.] If $n-1-d>2$ is odd, compute $H_i(|P_{>\sigma}|; \ZZ) = 0$ for $1<i<(n-d)/2$, and terminate.

	\noindent -- Then by Poincar\'{e} duality and the universal coefficients theorem for cohomology, we also have $H_{j}(|P_{>\sigma}|; \ZZ) \cong H^{n-1-d-j}(|P_{>\sigma}|; \ZZ) = 0$ for $(n-d)/2 \leq j < n-d-1$.

	\item[2\text{c}--ii.] If $n-1-d>2$ is even, compute $H_i(|P_{>\sigma}|; \ZZ) = 0$ for $1<i<(n-1-d)/2$ and $H_{(n-1-d)/2}(|P_{>\sigma}|; \ZZ) = 0$ \textit{or} $\chi(P_{>\sigma}) = 2$, and terminate.

	\noindent -- As in (3a), we then also have $H_{j}(|P_{>\sigma}|; \ZZ) \cong H^{n-d-1-j}(|P_{>\sigma}|; \ZZ) = 0$ for $(n-1-d)/2 < j < n-d-1$. To show that the middle homology group $H_{(n-1-d)/2}(|P_{>\sigma}|; \ZZ)=0$ using that $\chi(P_{>\sigma}) = 2$, we can argue as follows: let $r[p] = \rank H_{(n-1-d)/2}(|P_{>\sigma}|; \FF_p)$ for any prime $p$. Then $\chi(P_{>\sigma}) = 2-r[p] = 2$, hence $r[p] = 0$ and $H_{(n-1-d)/2}(|P_{>\sigma}|; \ZZ) = 0$.
	
\end{enumerate}

In particular, we emphasize that no linear algebraic computation (in the sense of computing Smith normal form) is necessary in the case $d \geq n-3$.
\end{rem}

\begin{rem}[Computing the fundamental category] \label{rem:categoryOfStrata} Let $A$ be a poset and $\Pi: \sC \to A$ be a $A$-stratified space, i.e. a conservative functor (Remark \ref{rem:localization}). Then we have the homotopy category $\Pi_1: h_1 \sC \to A$ of $\sC$, where every mapping space $\Map_{\sC}(x,y)$ is replaced by its set of connected components (in the $\infty$-categorical setup, $h_1 \sC$ should be thought of as the \emph{fundamental category} \cite{woolfFund} of the stratified space). Further collapsing every nonempty hom-set in $h_1 \sC$ to a single point yields a poset $h_0 \sC \to A$. $h_1 \sC$ and $h_0 \sC$ are the bottom two stages of the stratified Postnikov tower of $\Pi$ \cite[\S 2.3]{Exodromy}.

Given $n = \dim(P)$ and $\Pi^{\can}: \sP^{\can} \to [n]$, we can attempt to compute $h_1 \sP^{\can}$ and $h_0 \sP^{\can}$ given only the canonical stratification $\pi: P \to [n]$. To compute $h_0 \sC$,
\begin{enumerate} \item Compute the connected components of each fiber $P_{\pi = i}$. These then form the objects $[\sigma]$ of $h_0 \sC$, where $[\sigma]$ denotes equivalence classes of objects $\sigma \in P$.
\item Given $[\sigma]$ and $[\tau]$ such that $\pi(\sigma) < \pi(\tau)$, we have $[\sigma] < [\tau]$ in $h_0 \sC$ if and only if there exists a choice of representatives $\sigma \in [\sigma]$ and $\tau \in [\tau]$ such that $\sigma < \tau$ in $P$.
\end{enumerate}

As for $h_1 \sC$, suppose given objects $[\sigma]$ and $[\tau]$:
\begin{enumerate}
	\item If $[\sigma] = [\tau]$, then the computation of $\Hom_{h_1 \sC}([\sigma], [\sigma])$ amounts to a fundamental group calculation, for which we can give a generators and relations presentation using standard methods (which is not really satisfactory).
	\item On the other hand, suppose $[\sigma] < [\tau]$ in $h_0(\sC)$ with $\pi(\sigma) = i$ and $\pi(\tau) = j$. Let
    $$P_{\pi \in \{i,j\}} \coloneq \{ i < j\} \times_{[n]} P.$$
    Then $\Hom_{h_1 \sC}([\sigma], [\tau])$ is given by the subset of connected components $[\sigma' < \tau']$ in the poset of sections $\Fun_{/\{i<j\}}(\{i<j\},P_{\pi \in \{i,j\}})$ such that $\sigma' \in [\sigma]$ and $\tau' \in [\tau]$.
\end{enumerate}
\end{rem}

\subsection{Properties}

In the remainder of this section, we collect a few further theoretical observations concerning stratification. None of these results will be needed for our implementation of the stratification algorithm in \S \ref{sec:st-alg}. Our first result is the combinatorial avatar of the following fact concerning manifolds with boundary: given a compact smooth manifold $M$ with boundary $\partial M$, the inclusion $M - \partial M \to M$ is a homotopy equivalence because of the existence of a collar neighborhood of $\partial M$.

\begin{prp} Let $P$ be a finite $\delta$-admissible poset, let $G$ be its generic strata, and let $U \supset G$ be a cosieve in $P$ that contains $G$. Suppose that $(\LL_P)|_{U-G} = 0$. Then the inclusion $G \to U$ is cofinal.
\end{prp}
\begin{proof} Given $\sigma \in U$, let $l(\sigma)$ be the minimum length $l$ taken across all chains
\[ \sigma = \tau_0 < \tau_1 < \tau_2 < \ldots < \tau_l \]
such that $\dim(\tau_{i+1}) = \dim(\tau_i) + 1$ and $\tau_l \in G$, and let $l(U) = \max\{ l(\sigma): \sigma \in U \}$.

Our strategy is to proceed by induction on $l(U)$ and use Quillen's Theorem A (\cite[Thm.~4.1.3.1]{HTT}). If $l(U) = 0$, then $G = U$ and there is nothing to prove. If $l(U) = 1$, then for every $\sigma \in U - G$, we have that $G \times_U U_{\geq \sigma} = P_{>\sigma}$, which is weakly contractible by the hypothesis that $\LL_P(\sigma) = 0$ and the identification $\LL_P(\sigma) \simeq \Sigma^{\infty} S^d |P_{>\sigma}|$ of Theorem \ref{thm:shiftingLink}. Invoking Quillen's Theorem A then completes the proof in this case.

Now suppose that the claim is proven for all triples $(P,G,U)$ as in the theorem statement with $l(U) \leq l$, and suppose $l(U) = l+1$. Let $\sigma \in U - G$. Let $G_{> \sigma} \coloneq G \times_U U_{\geq \sigma} = G \cap U_{\geq \sigma}$. We want to show that $G_{> \sigma}$ is weakly contractible. Equivalently, by our assumption that $\LL_P(\sigma) = 0$, $P_{> \sigma}$ is weakly contractible, so it is enough to show that the cosieve inclusion $G_{> \sigma} \to U_{> \sigma} = P_{> \sigma}$ is cofinal. Consider the triple $(P_{> \sigma}, G_{> \sigma}, P_{> \sigma})$. By Corollary \ref{cor:LinkAdmissible}, $P_{> \sigma}$ is $\delta$-admissible. Moreover, by Proposition \ref{prp:shiftedOrientationSheaf},
\[ \LL_{P_{> \sigma}} \simeq \Sigma^{-(d+1)} \left( \left. \LL_P \right|_{P_{> \sigma}} \right), \]
so the restriction of $\LL_{P_{> \sigma}}$ to $G_{> \sigma}$, resp. $P_{> \sigma} - G_{> \sigma}$ is locally constant at $\Sigma^{\infty} S^{n-1-d}$, resp. 0. Now because $l(P_{>\sigma}) \leq l$, we are done by induction. 
\end{proof}

\begin{nul}[Lefschetz-Poincar\'{e} duality for the generic strata] \label{PoincareDualityGenericStrata} Let $G$ be the generic strata of $P$. Let $G_{\geq d}$, resp. $G_d$ be the subposets of $G$ consisting of $\sigma$ with $\dim_P(\sigma) \geq d$, resp. $\dim_P(\sigma) = d$. Then $G$ admits a filtration
\[ \emptyset \subset G_{\geq n} \subset G_{\geq n-1} \subset \ldots \subset G_{\geq 0} = G \]
where for each inclusion $G_{\geq d+1} \subset G_{\geq d}$ we have the pushout square
\[ \begin{tikzcd}[row sep=4ex, column sep=4ex, text height=1.5ex, text depth=.5ex]
\bigsqcup_{\sigma \in G_d} P_{> \sigma} \ar{r} \ar{d} & G_{\geq d+1} \ar{d} \\
\bigsqcup_{\sigma \in G_d} P_{\geq \sigma} \ar{r} & G_{\geq d}
\end{tikzcd} \]
obtained by iterative application of Theorem \ref{thm:MayerVietoris} over all $\sigma \in G_d$. Upon applying the functor $\Sigma^{\infty}_+ \left| - \right|$, we have the pushout square
\[ \begin{tikzcd}[row sep=4ex, column sep=4ex, text height=1.5ex, text depth=.5ex]
\bigsqcup_{\sigma \in G_d} \Sigma^\infty_+ |S^{n-1-d}|  \ar{r} \ar{d} & \Sigma^\infty_+ |G_{\geq d+1}| \ar{d} \\
\bigsqcup_{\sigma \in G_d} \Sigma^\infty_+ |D^{n-d}| \ar{r} & \Sigma^\infty_+|G_{\geq d}|,
\end{tikzcd} \]
thereby obtaining a \emph{stable cell decomposition} of $\Sigma^\infty_+|G|$, which is dual to the unstable cell decomposition of $|P|$ defined by its $\delta$-structure in the sense that every $\sigma \in G$ of $P$-dimension $d$ corresponds to a cell of $G^{\op}$-dimension $n-d$. Taking cellular homology with coefficients in a ring $R$, we obtain a cochain complex $D^{\bullet}$
\[ \bigoplus_{\sigma_0 \in G_0} R \xto{\partial^0} \bigoplus_{\sigma_1 \in G_1} R \xto{\partial^{1}} \ldots \xto{\partial^{n-2}} \bigoplus_{\sigma_{n-1} \in G_{n-1}} R \xto{\partial^{n-1}} \bigoplus_{\sigma_n \in G_n} R \]
such that $H^i(D) \cong H_{n-i}(|G|;R)$.

We can go further and prove a Lefschetz-Poincar\'{e} duality result in our setting. To formulate this, let $L$, resp. $K$ be the semisimplicial sets $(\Delta^{\inj})^{\op} \to \Set$ that as functors classify the discrete cartesian fibrations $\delta|_{P-G}: P-G \to \Delta^{\inj}$, resp. $\delta: P \to \Delta^{\inj}$. Let $C_{\ast}(K,L)$ be the relative chain complex under the Dold-Kan correspondence, taken with $R$ coefficients, so $C_{\ast}(K,L)$ equals
\[ \bigoplus_{\sigma_n \in G_n} R \xto{\partial_n} \bigoplus_{\sigma_{n-1} \in G_{n-1}} R \xto{\partial_{n-1}} \ldots ... \xto{\partial_2} \bigoplus_{\sigma_1 \in G_1} R \xto{\partial_1} \bigoplus_{\sigma_0 \in G_0} R. \]
Then by the same analysis as in \ref{PoincareDuality}, we identify the cochain differential $\partial^d$ as, up to sign, the transpose of $\partial_{d+1}$. Taking $R = \FF_2$, we deduce that $H_{n-i}(|G|;\FF_2) \cong \widetilde{H}^i(|P|/|P-G|;\FF_2)$. Moreover, if $(\LL_P)|_G$ is orientable in the sense that the monodromy action on $\Sigma^\infty S^n$ is trivial, then we have duality at the level of stable homotopy:
\[ \Sigma^\infty_+ |P|/|P-G| \simeq \varprojlim_{\sigma \in G} \LL_P(\sigma) \simeq F(\Sigma^{\infty}_+|G|,\Sigma^\infty S^n). \]
\end{nul}

\begin{nul}[Shriek pullback vs. star pullback] \label{shriekVsStar} Let $F: P \to \Sp$ be a sheaf on $P$. Then there are at least two reasonable definitions for a map of posets $\pi: P \to Q$ to be a $F$-stratification of $P$:
\begin{enumerate}
	\item For every $x \in Q$, the restriction of $F$ to the fiber $P_{\pi = x} \coloneq \{ x\} \times_Q P$ is locally constant.
	\item For every $x \in Q$, let $i_x: P_{\pi \leq x} \coloneq Q_{\leq x} \times_Q P \to P$ denote the sieve inclusion. Then $((i_x)^! F)|_{P_x}$ is locally constant, where $(i_x)^!$ is defined as in \ref{shriekFiberSequence}.
\end{enumerate}

 Taking $F = \LL_P$, we note that the map $\pi: P \to [n]$ of Definition \ref{dfn:strat} satisfies the second condition but not generally the first. To explain, recall from Proposition \ref{prp:shriekFunctoriality} that given a sieve inclusion $i: Q \to P$ with complementary cosieve inclusion $j: P-Q \to P$, we have $\LL_Q \simeq i^! \LL_P$ and the resulting fiber sequence
\[ \LL_Q \simeq i^! \LL_P \to \left. \left( \LL_P \right) \right|_Q = i^\ast \LL_P \to \left. \left( j_{\ast} (\left. \LL_P \right|_{P-Q}) \right) \right|_{Q} = (i^\ast j_\ast j^\ast)(\LL_P). \]

Suppose that $P-Q = G$ is the generic strata. Then even though $(\LL_P)|_G$ is locally constant, it may fail to be the case that $(i^\ast j_\ast)(\LL_P|_G)$ is locally constant, so the question of local constancy of $\LL_{P-G}$ differs from that of $(\LL_P)|_{P-G}$. For example, consider the poset of simplices of the ordered simplicial complex with vertices
\[ \{0,1,2,3,4,5,6,7, 8 \} \]
and simplices
\[ \{(0,1,3),(0,2,3), (1,3,5), (2,3,4), (2,4,6), (3,4,5), (4,5,7), (4,6,7), (3,4,8) \}. \]

Then upon removal of the generic strata, we have the $1$-dimensional sub-simplicial complex which is the disjoint union of two circles
\[ \begin{tikzcd}[row sep=2ex, column sep=2ex, text height=1.5ex, text depth=0.25ex]
& 0 \ar{r} \ar{dl} & 1 \ar{dr} \\
2 \ar{dr} & & & 5 \ar{dl} \\
& 6 \ar{r} & 7
\end{tikzcd}, \quad
\begin{tikzcd}[row sep=2ex, column sep=2ex, text height=1.5ex, text depth=0.25ex]
3 \ar{rd} \ar{dd}  \\
& 8 \\
4 \ar{ru}
\end{tikzcd}. \]

Therefore, after collapsing connected components of each strata to points, we have the poset
\[ \begin{tikzcd}[row sep=4ex, column sep=4ex, text height=1.5ex, text depth=0.25ex]
\bullet \ar{r} & \bullet \\
\bullet \ar{ru} \ar{r} & \bullet
\end{tikzcd} \]
over $\{1<2\} \subset [2]$. However, one can compute $\LL^{\ZZ}_P(3) \to \LL^{\ZZ}_P(34)$ to not be a local homology equivalence; indeed, $\pi_2(\LL^{\ZZ}_P(3))$ has rank 1 while $\pi_2(\LL^{\ZZ}_P(34))$ has rank 2.
\end{nul}



\begin{nul}[Functoriality of the generic strata] In general, the canonical stratification is not functorial with respect to maps of posets. However, we can at least say the following.

\begin{lem} \label{lm:genericStratumFunctoriality} Let $i: Q \to P$ be a sieve inclusion of $\delta$-admissible posets with $n = \dim(Q) = \dim(P)$, let $G_Q$ resp. $G_P$ be the generic strata of $Q$ resp. $P$, and let $\sigma \in Q$. Suppose that $\sigma$ is in $G_P$. Then $\sigma$ is in $G_Q$ if and only if $Q_{>\sigma} = P_{> \sigma}$.
\end{lem}
\begin{proof} First note that $\dim_Q(\sigma) = \dim_P(\sigma)$ because $i$ is a sieve inclusion; let $d$ denote this common dimension. Because $\dim(Q) = \dim(P)$, the claim is obviously true if $d=n$, so let us suppose $d<n$. Our assumption that $\sigma \in G_P$ is equivalent to the two conditions:
\begin{enumerate}
	\item For all $\tau > \sigma$, we have that $\tau \in G_P$.
	\item $\Sigma^{\infty}_+ |P_{>\sigma}| \simeq \Sigma^{\infty}_+ S^{n-1-d}$.
\end{enumerate}

Let us suppose as an inductive hypothesis that the claim holds for all $\tau > \sigma$. For the ``if'' statement, suppose that $Q_{>\sigma} = P_{>\sigma}$. Then for all $\tau > \sigma$, $Q_{>\tau} = P_{>\tau}$, so by induction $\tau \in G_Q$. Consequently, because $\Sigma^{\infty}_+|Q_{>\sigma}| \simeq \Sigma^{\infty}_+|P_{>\sigma}| \simeq \Sigma^{\infty}_+ S^{n-1-d}$, we get that $\sigma \in G_P$.

Conversely, for the ``only if'' statement, suppose that $\sigma \in G_Q$. Then for all $\tau > \sigma$ with $\tau \in Q$, by induction we have that $Q_{>\tau} = P_{>\tau}$. Thus, $Q_{>\sigma} \subset P_{>\sigma}$ is a cosieve. But $Q_{>\sigma} \subset P_{>\sigma}$ is also a sieve because $i$ is. A subposet is both a sieve and a cosieve if and only if it is a connected component. Thus, we get that $P_{>\sigma} = Q_{>\sigma} \sqcup R$, and
\[  \Sigma^{\infty}_+ |P_{>\sigma}| \simeq \Sigma^{\infty}_+ |Q_{>\sigma}| \oplus \Sigma^{\infty}_+ |R|. \]
Because both $\sigma \in G_P$ and $\sigma \in G_Q$, the map $\Sigma^{\infty}_+ |Q_{>\sigma}| \to \Sigma^{\infty}_+ |P_{>\sigma}|$ is an equivalence with cofiber $\Sigma^{\infty}_+ |R|$, so we moreover have that $\Sigma^{\infty}_+ |R| \simeq 0$, which forces $R = \emptyset$ and $P_{>\sigma} = Q_{>\sigma}$.
\end{proof}

Now suppose that we have a filtration by sieve inclusions
\[ P_0 \to P_1 \to \ldots \to P_{m-1} \to P_m = P \]
with $\dim(P_i) = \dim(P)$ for all $i$ and $P$ equal to its own generic strata (e.g., the poset of simplices of a triangulation of a closed manifold). Let $G_i$ be the generic strata of $P_i$. Then by Lemma \ref{lm:genericStratumFunctoriality}, for $\sigma \in P_i$, we have $\sigma \in G_i$ if and only if $(P_{i})_{>\sigma} = P_{>\sigma}$, so in particular if $\sigma \in G_i$ then $\sigma \in G_j$ for all $j \geq i$. We thus obtain a filtration of the generic strata
\[ G_0 \to G_1 \to \ldots \to G_{m-1} \to G_m = P. \]

As a central computational tool in applied topology, one has the persistent homology of a filtered simplicial complex, with the filtration typically defined by varying a scaling parameter; at maximal scale, one is left with a manifold or manifold with boundary. As a variant, we propose to instead compute the persistent homology of the induced filtration of generic strata, with the expectation that, for certain applications, the contributions of the non-generic strata to persistent homology are undesirable and should be discarded.
\end{nul}

\clearpage

\section{Stratification algorithm}\label{sec:st-alg}

In this section, we discuss some of the implementation-level details and present pseudocode for the canonical stratification algorithm as described in Algorithm \ref{algorithm}. We first discuss the algorithm in the case of arbitrary dimension. Distinct implementations for special lower-dimensional cases are then discussed in \S \ref{sec:3d-alg}.


Let $K$ be a finite $n$-dimensional abstract simplicial complex, let $P$ be its poset of simplices, and let $\pi: P \to [n]$ be the canonical stratification (Definition \ref{dfn:strat}). The \emph{$i$-strata} was defined to be the subposet $P_{\pi=i} \coloneq \{i\} \times_{[n]} P$. In general, $P_{\pi=i}$ consists of multiple connected components, which we are interested in determining. Let us call a connected component of $P_{\pi=i}$ an \emph{$i$-stratum}.

Recall from \ref{algorithm} that the proposed algorithm is an iterative algorithm which accepts $K$ as an initial input. On the first iteration, simplices lying in the $n$-strata $P_{\pi=n}$ are identified and assigned to their particular $n$-stratum. Then the simplices in the $n$-strata are removed, with the remaining simplices forming a subcomplex $K^1$ of lower dimension $\dim(K^1) < \dim (K)$. $K^1$ is then used as the simplicial complex for the next iteration. For each subsequent iteration, this procedure is repeated, and the algorithm terminates after assigning each simplex to the $i$-stratum to which it belongs.

To reduce the potential for confusion, we henceforth adopt the following convention:
\begin{itemize}
    \item[($\ast$)] We call the dimension of the subcomplex at a particular iteration the \emph{top dimension} and denote it by $n_{\cur}$. 
\end{itemize}

\noindent We will also make use of the following terminology:

\begin{itemize}
\item[($\ast$)] Given a face-coface pair $\sigma \subset \tau$ such that $\dim(\tau) - \dim(\sigma) = 1$, we call $\tau$ an \emph{immediate coface} of $\sigma$ and $\sigma$ an \emph{immediate face} of $\tau$.
\item[($\ast$)] We define the \emph{codimension} of a simplex to be the difference between its dimension and $n_{\cur}$; for instance, the codimension of a $d$-simplex is $k = n_{\cur} - d$. $c_k$-$simplices$ are simplices of codimension $k$.
\end{itemize}

Algorithm \ref{alg:main-alg} shows pseudocode containing the general structure of the algorithm.

\begin{algorithm}[H]
\caption{Canonical Stratification}\label{alg:main-alg}
\begin{algorithmic}[1]
\For{$n_{\cur} \gets n \text{ \textbf{to} } 1 \text{ \textbf{step} } -1$}
\If {\texttt{complex has at least one $c_0\text{-}simplex$}}
\State \texttt{assign all $c_0\text{-}simplices$ to individual strata}
\For{$k \gets 1 \text{ \textbf{to} } n_{\cur}$}
\ForAll{\texttt{$simplex$ $\in$ $c_k\text{-}simplices$}}
\If{\texttt{$simplex$ belongs to exactly one $n_{\cur}$-stratum}}
\State \texttt{Add $simplex$ to the $n_{\cur}$-stratum}
\ElsIf{\texttt{$simplex$ belongs to multiple $n_{\cur}$-strata}}
\State \texttt{Merge strata into a single stratum}
\State \texttt{Add $simplex$ to merged $n_{\cur}$-strata}
\EndIf
\EndFor
\EndFor
\State \texttt{Remove $n_{\cur}$-strata}
\EndIf
\EndFor
\end{algorithmic}
\end{algorithm}


A simplex belongs to the $n_{\cur}$-strata, specifically to the stratum of its immediate cofaces, if the following conditions are met:
\begin{enumerate}
\item All of the immediate cofaces of the simplex lie in the $n_{\cur}$-strata. Moreover, except for the $c_1$-simplex case, all of the immediate cofaces of the simplex lie in the same $n_{\cur}$-stratum.
\item The homology of the small link of the simplex equals the homology of a ($k-1$)-sphere. 
\end{enumerate}

Condition (1) is tested first because it is computationally cheaper. Pseudocode for this is shown in Algorithm \ref{alg:uniquecoface}. Note also that the connectedness check is exempted for $c_1$-simplices. Consequently, $c_1$-simplices are the only simplices that can belong to the $n_{\cur}$-strata while having cofaces that belong to multiple $n_{\cur}$-strata, so this is the only point where any merging of strata can take place. Therefore, we can combine the $c_1$-simplex case with the $c_0$-simplex case to entirely avoid merging of strata. We discuss this in \S \ref{sec:codim01}. 

If condition (1) is met, condition (2) is then tested. Recall that the small link of a simplex is the set of all cofaces of the simplex (Definition \ref{dfn:link}). Pseudocode for finding the small link is shown in Algorithm \ref{alg:getsmalllink}.

Finally, we follow the procedure outlined in Remark \ref{rem:reduceComputationPoincareDuality}. If codimension $\leq 3$, then there are tricks for checking $n_{\cur}$-strata membership that do not involve computing homology. We break out these cases as separate algorithms, described in \S \ref{sec:3d-alg}. If codimension $>3$, the algorithm computes some integral homology groups to be zero. Specifically, given the small link of $\sigma$, the algorithm constructs the chain complex
\[ \begin{tikzcd}[row sep=4ex, column sep=4ex, text height=1.5ex, text depth=2ex]
\bigoplus\limits_{\shortstack{$\scriptstyle \tau > \sigma,$ \\ $\scriptstyle \dim(\tau)-\dim(\sigma)$ \\ $\scriptstyle = \left \lceil k/2 \right \rceil + 1$ }} \ZZ \ar{r} & \bigoplus\limits_{\shortstack{$\scriptstyle \tau > \sigma,$ \\ $\scriptstyle \dim(\tau)-\dim(\sigma)$ \\ $\scriptstyle = \left \lceil k/2 \right \rceil$}} \ZZ \ar{r} &  \ldots \ar{r} & \bigoplus\limits_{\shortstack{$\scriptstyle \tau > \sigma,$ \\ $\scriptstyle \dim(\tau)-\dim(\sigma)$ \\ $\scriptstyle = 2$}} \ZZ \ar{r} & \bigoplus\limits_{\shortstack{$\scriptstyle \tau > \sigma$ \\ $\scriptstyle \dim(\tau)-\dim(\sigma)$ \\ $\scriptstyle = 1$}} \ZZ
\end{tikzcd} \]
\\

and computes the homology to be zero except at the ends (where homology is not computed). We note here that to define the boundary maps of the chain complex, we should be given a fixed global ordering of the vertices of the simplicial complex. Computing homology is a well-documented procedure and is implemented in multiple TDA libraries, so we will not enter into a deeper discussion of the mechanics of this step here.

\subsection{Data structures}\label{sec:data-structure}

The simplicial complex is represented as a graph with each simplex as a node. Each simplex only knows its immediate cofaces and its immediate faces. Note that the graph for an $n$-complex is a multi-partite graph with $n$ partitions, with one partition for every simplex dimension. Furthermore, to improve access time the simplices of an $n$-complex are stored as objects in $n$ many lists, one for each dimension.

There are other, more memory efficient storage formats like the simplex tree \cite{Boissonnat2014}. However, many parts of the algorithm require performing a graph traversal as well as quick access to all simplices in a certain dimension. Therefore, we use a structure that allows better access during runtime instead of a memory efficient data structure.

The membership of simplices in strata is stored as a map, $M | simplex \to strata$. We choose the map $M$ because the majority of lookups in the algorithm are $simplex \to strata$. Note that constructing the sets of members for each strata from $M$ is a linear operation.

To remove $n_{\cur}$-strata, we keep and update a list of the cofaces of each simplex in the remaining subcomplex. This approach is used instead of removing objects from lists and trees for computational performance and to avoid modifying the input. At the end of each iteration in $n$, this list is updated by removing members of the list. More memory efficient but slower alternatives would involve computing the list of cofaces on the fly, or removing simplices from the list and trees.

Finally, let us note that we have not opted to store any of the finer structure afforded by the canonical stratification, such as the poset or category structure on the set of strata as discussed in Remark \ref{rem:categoryOfStrata}. 

\subsection{General subroutines} \label{sec:subroutines}

\subsubsection{Unique $n_{\cur}$-stratum subroutine}
This subroutine checks if all of the immediate cofaces of a simplex lie in the same strata.
For every immediate coface of the simplex, there are only three cases.
\begin{itemize}
\item The coface is unassigned: then break the loop as the simplex is not in the $n_{\cur}$-strata.
\item The coface is assigned but different from the previous: then break the loop as the simplex is not in the $n_{\cur}$-strata.
\item The coface is assigned and the same as the previous: continue.
\end{itemize}

Algorithm \ref{alg:uniquecoface} shows pseudocode for this schema, returning the unique stratum if it exists or $NULL$ otherwise.
\begin{algorithm}[H]
\caption{Finding unique stratum among immediate cofaces}\label{alg:uniquecoface}
\begin{algorithmic}[1]
\State $M \gets $ \texttt{Map from simplex to stratum. Default $NULL$. }
\newline
\Procedure{uniqueStratumAmongCofaces}{simplex}
  \State $stratum \gets NULL$
  \ForAll{$coface \in simplex.getImmediateCofaces()$}
    \If{$M[coface] = NULL$}
      \State $stratum \gets NULL$
      \State \bf{break}
    \ElsIf{$stratum \neq NULL$ \bf{and} $stratum \neq M[coface]$}
      \State $stratum \gets NULL$
      \State \bf{break}
    \Else{}
      \State $stratum \gets M[coface]$
    \EndIf
  \EndFor
  \State \Return $stratum$
\EndProcedure
\end{algorithmic}
\end{algorithm}

\subsubsection{Small link subroutine}

This subroutine returns the small link of a given simplex by a procedure similar to that of a connected component analysis in a directed graph. The small link is found by following the cofaces recursively and constructing a set. There are multiple paths to the same coface, so the recursion should terminate on cofaces that are already in the set. For easier access, the set is organized by the relative dimension. Finally, the set only needs to be constructed to store cofaces up to the relative dimension $\left \lceil k/2 \right \rceil + 1$, where $k=n_{\cur} -d$. Note that the indexing is such that $SL[i]$ is the set of cofaces $\tau$ of the given simplex $\sigma$ with $\dim(\tau) - \dim(\sigma) = i+1$.

Algorithm \ref{alg:getsmalllink} shows pseudocode that implements this procedure.

\begin{algorithm}[H]
\caption{Finding small link}\label{alg:getsmalllink}
\begin{algorithmic}[1]
\State $n_{\cur} \gets $ \texttt{current top dimension}
\newline
\Procedure{getSmallLink}{$simplex$}
  \State $d \gets simplex.getDimension()$
  \State $SL \gets $\texttt{array of } $ceil((n_{\cur}-d)/2)$ \texttt{empty sets}
  \State $addCofaces(SL, simplex, 0)$
  \State \Return $SL$
\EndProcedure
\newline
\Procedure{addCofaces}{$SL, simplex, sl\_dim$}
  \ForAll{$coface \in simplex.getImmediateCofaces()$}
    \If{$coface \notin SL[sl\_dim]$}
      \State $SL[sl\_dim].add(coface)$
      \If{$sl\_dim < SL.size()-1$}
        \State $addCofaces(SL, coface, sl\_dim+1)$
      \EndIf
    \EndIf
  \EndFor
\EndProcedure
\end{algorithmic}
\end{algorithm} 

\subsection{The case of codimension $\leq 3$}\label{sec:3d-alg}


\subsubsection{Codimension 0/1} \label{sec:codim01}

A $c_1$-simplex is in the $n_{\cur}$-strata if and only if it has exactly two cofaces. Furthermore, the strata of two $c_0$-simplices are merged if they share a $c_1$-simplex. Given these conditions, the codimension 0/1 case can be reduced to a connected components search together with an additional condition.

First, recall that the simplices are stored in a multi-partite graph.
In particular, $c_0$-simplices and $c_1$-simplices form a bipartite graph. 
We can then collapse this bipartite graph and consider each $c_1$-simplex as a node in a graph with cofaces determining the edges. Finding a connected component in the original bipartite graph is then equivalent to finding that connected component in the collapsed graph.

Now add the additional condition that two nodes sharing an edge are connected if and only if they are both in the $n_{\cur}$-strata.
Furthermore, any coface of a $c_1$-simplex in the $n_{\cur}$-strata is considered ``connected'' to the simplex.
A connected component of one node is counted only if the $c_1$-simplex is in the top stratum, and it include the cofaces of the simplex.
Both the nodes (the $c_1$-simplices) and the $c_0$-simplices in each connected component lie in the same $n_{\cur}$-stratum. Finally, all leftover $c_0$-simplices are assigned to individual strata.

There are multiple well-known implementations of a connected components search. Here, we give a depth-first recursive algorithm.
At each node, test the following:
\begin{enumerate}[topsep=2pt]
\item $terminate$ if the node is assigned.
\item $terminate$ if the node is not in the $n_{\cur}$-strata.
\end{enumerate}
If unterminated, the node and its cofaces are added to the stratum, and the recursive check then continues to all nodes connected to the simplex.

Algorithm \ref{alg:connectedcomponent} shows pseudocode for this recursive search.
\begin{algorithm}[H]
\caption{Recursive Connected Component Search}\label{alg:connectedcomponent}
\begin{algorithmic}[1]
\State $M \gets $ \texttt{Map from simplex to stratum. Default $NULL$. }
\newline
\Procedure{connectedComponentSearch}{$stratum,simplex$}
\State $CC \gets simplex.getImmediateCofaces()$
\If{$M[simplex] = NULL$ \bf{and} $ CC.size() = 2$}
\State $M[simplex] \gets stratum$
\ForAll{$coface \in CC$}
\State $M[coface] \gets stratum$
\ForAll{$face \in coface.getImmediateFaces()$}
\State $connectedComponentSearch(stratum,face)$
\EndFor
\EndFor
\EndIf
\EndProcedure
\end{algorithmic}
\end{algorithm}
Generally, a depth-first connected components search algorithm needs to store nodes that have been visited to avoid getting stuck in an infinite loop. However, the above two termination conditions turn out to be sufficient to replace a check for node visitation. To see this, note that:
\begin{itemize}[topsep=2pt]
\item A node in the $n_{\cur}$-strata will have a stratum assigned on the first visit, and condition (1) will then cause termination on subsequent visits.
\item A node that is not in the $n_{\cur}$-strata will always terminate by condition (2), so it does not require the visited check.
\end{itemize}

With that said, it will be more computationally efficient to store visited nodes because the check for $n_{\cur}$-strata membership may not be O(1).
\\

The main part of the codimension $1$ code iterates through every $c_1$-simplex, and calls Algorithm \ref{alg:connectedcomponent} when it finds a simplex in the $n_{\cur}$-strata. Afterwards, remaining $c_0$-simplices are assigned to individual $n_{\cur}$-strata.
\\

Algorithm \ref{alg:codim1} shows pseudocode for the codimension $1$ case.
\begin{algorithm}[H]
\caption{Codimension 1 case}\label{alg:codim1}
\begin{algorithmic}[1]
\State $M \gets $ \texttt{Map from simplex to stratum. Default $NULL$. }
\State $n_{\cur} \gets $ \texttt{current top dimension}
\newline
\Procedure{codimOneCase}{}
\ForAll{$simplex \in c_1\text{-}simplices$ }
\State $CC \gets simplex.getImmediateCofaces()$
\If{$M[simplex] = NULL$ \bf{and} $ CC.size() = 2$}
\State $stratum \gets addNewStratum(top\_dimension \gets n_{\cur})$
\State $connectedComponentSearch(stratum,simplex)$
\EndIf
\EndFor
\ForAll{$simplex \in c_0\text{-}simplices$ }
\If{$M[simplex] = NULL$}
\State $stratum \gets addNewStratum(top\_dimension \gets n_{\cur})$
\State $M[simplex] \gets stratum$
\EndIf
\EndFor
\EndProcedure
\end{algorithmic}
\end{algorithm}
Here, \texttt{addNewStratum} creates a new stratum.

\subsubsection{Codimension 2}

To check if a $c_2$-simplex is in the $n_{\cur}$-strata, it suffices to test if all its cofaces lie in the same $n_{\cur}$-stratum. The algorithm for codimension $2$ thus simply iterates through the $c_2$-simplices and checks if all of its cofaces are in the same $n_{\cur}$-stratum.
\\

Algorithm \ref{alg:codim2} shows pseudocode for the codimension 2 case. 
\begin{algorithm}[H]
\caption{Codimension 2 case}\label{alg:codim2}
\begin{algorithmic}[1]
\State $M \gets $ \texttt{Map from simplex to stratum. Default $NULL$. }
\newline
\Procedure{codimTwoCase}{}
  \ForAll{$simplex \in c_2\text{-}simplices$ }
    \If{$M[simplex] = NULL$}
      \State $stratum \gets uniqueStratumAmongCofaces(simplex)$
      \If{$stratum \neq NULL$}
        \State $M[simplex] \gets stratum$
      \EndIf
    \EndIf
  \EndFor
\EndProcedure
\end{algorithmic}
\end{algorithm}
Any $c_2$-simplex that is unassigned is left for the next iteration.

\subsubsection{Codimension 3}

For a $c_3$-simplex, we have to first check if all of its immediate cofaces lie in the same $n_{\cur}$-stratum, and then check if the Euler characteristic $\chi$ of its small link is $2$. Here, $\chi \coloneq V - E + F$, where $V$, $E$ and $F$ is the number of $c_2$, $c_1$ and $c_0$-simplices in the small link, respectively.
\\

Algorithm \ref{alg:codim3} shows pseudocode for the codimension $3$ case.
\begin{algorithm}[H]
\caption{Codimension 3 case}\label{alg:codim3}
\begin{algorithmic}[1]
\State $M \gets $ \texttt{Map from simplex to stratum. Default $NULL$. }
\newline
\Procedure{codimThreeCase}{}
  \ForAll{$simplex \in c_3\text{-}simplices$ }
    \If{$M[simplex] = NULL$}
      \State $stratum \gets uniqueStratumAmongCofaces(simplex)$
      \If{$stratum \neq NULL$}
        \State $SL \gets getSmallLink(simplex)$
        \State $x \gets SL[0].size()-SL[1].size()+SL[2].size()$
        \If{$x = 2$}
          \State $M[simplex] \gets stratum$
        \EndIf
      \EndIf
    \EndIf
  \EndFor
\EndProcedure
\end{algorithmic}
\end{algorithm}
Any $c_3$-simplex that is unassigned is left for the next iteration.

\subsection{Time complexity}

The canonical stratification takes an $n$-complex as input, so this subsection will present the time complexity in terms of size of this $n$-complex. To do this, let us denote the number of simplices of dimension $d$ by $s_d$ and the total number of simplices by $s$. We first discuss the time complexity in general and then specialize to the $n \leq 3$ case. In addition, at the end we briefly discuss how time complexity scales with the number $s_0$ of $0$-simplices. This is useful for situations where the input $n$-complex is constructed as a sort of auxiliary structure on top of the $0$-simplices, which represent the data points of interest.

\subsubsection{General case}

For the general case, the algorithm computes the integral homology of the small links. The running time for computing integral homology of a chain complex is dominated by the operation of computing Smith normal form over $\ZZ$. There exist many algorithms of varying time complexity for computing Smith normal form \cite{10.1007/978-3-662-05148-1_10}. Let us black box this running time as O($\Theta(s)$). The number of needed homology group calculations scales as O($s$). To conclude, the time complexity of the canonical stratification algorithm is O($s \cdot \Theta(s)$).

In many types of constructions of simplicial complexes in which the size scales but the dimension is fixed, the way in which new simplices are added is such that the small link size is O($1$). However, the small link size does generally change with the dimension of the input complex. If we fix the dimension $n$ and invoke the $O(1)$ assumption on the small link size, then the time complexity further reduces to O($s$).

\subsubsection{$(n \leq 3)$-complex case}

The following are the contributions of the three special cases as well as removing the $n_{\cur}$-strata:

\begin{itemize}
\item \textbf{Codimension 1}: A general connected component analysis similar to the one shown in Algorithm \ref{alg:connectedcomponent} is a O($V+E$) workload, where $V$ is the number of vertexes and $E$ is the number of edges. The additional check to see if a simplex has 2 cofaces can be done in O(1). Thus, the complexity is O($s$).

\item \textbf{Codimension 2}: A $d$-simplex has O($s_{d+1}$) cofaces, so the codimension 2 case has complexity O($s_d*s_{d+1}$). Thus, the complexity is O($s^2$).

\item \textbf{Codimension 3}: With $n \leq 3$, the codimension $3$ case only occurs with $c_3$-simplices being 0-simplices. The procedure for finding the small link of a 0-simplex is a variant of a connected components search, and scales as O($s$). This is repeated for each 0-simplex, so the final complexity is once again O($s^2$).

\item \textbf{Removing $n_{\cur}$-strata}: Once again, a $d$-simplex has O($s_{d+1}$) cofaces, so updating the list of remaining simplices also has complexity O($s^2$).
\end{itemize}

Putting it all together, the $(n \leq 3)$-complex case has time complexity O($s^2$). However, the above analysis uses the worst case scaling for the number of cofaces; that is, O($s_{d+1}$) for a $d$-simplex. If we assume that the number of cofaces for a $d$-simplex is O(1), then finding the small link will also be a O(1) operation. Thus, with this assumption the time complexity is further reduced to O($s$).

\subsubsection{On the size of a simplicial complex}

In practical applications, the simplicial complex of interest is often constructed from a set of data points acting as the $0$-simplices. Therefore, it is useful to determine how $s$ scales with the number of 0-simplices $s_0$ so that the complexity can be expressed in terms of $s_0$.
Though the exact scaling depends on the particular mechanism used for simplicial complex construction, a naive combinatorial upper bound can be found for $s_d$ supposing $d \ll s_0$:
\[  s_{d} \leq {s_0 \choose {d+1}} \approx {s_0}^{d+1} \]
We thus see that an arbitrary abstract simplicial complex of dimension $n$ has O(${s_0}^{n+1}$) simplices given $n \ll s_0$. Thus, the time complexity of the canonical stratification algorithm is O(${s_0}^{(3n+3)}$) for $n>3$ and O(${s_0}^{(2n+2)}$) for $n \leq 3$. 

Restricting the $n$-complex to one which admits a (piecewise linear) embedding into Euclidean space $\RR^m$ gives additional constraints that can reduce the size, depending on $m$ and $n$. In particular, we have (cf. \cite[\S 4]{Bjorner2017}):

\begin{itemize}
    \item If $n=1$ and $m=2$ (i.e., the simplicial complex is a planar graph), then $s_1 \leq 3 s_0 - 6$.
    \item If $n=2$ and $m=2$, then in addition $s_2 \leq 2 s_0-5$.
    \item If $n=2$ and $m=3$, then $s_2 \leq s_0 (s_0-3)$.
    \item If $n=3$ and $m=3$, then in addition $s_3 \leq s_0 (s_0 - 3)/2 -1$.
\end{itemize}

Thus, the upper bound of $s$ is O($s_0$) given an embedding in $\RR^2$ or O(${s_0}^2$) given an embedding in $\RR^3$, yielding a time complexity for the algorithm of O(${s_0}^2$), respectively O(${s_0}^4$). Finally, if we add the assumption that the number of cofaces for a simplex scales as O($1$), then the complexity is further reduced to O($s_0$), respectively O(${s_0}^2$).



\subsection{Experimental results}

We present some experimental results on the time complexity scaling for triangulations of the $2$-sphere and $3$-ball, displayed in Fig. \ref{fig:experiment}. The input simplicial complexes in both cases were constructed by means of the Delaunay triangulation. Specifically, triangulations of the $2$-sphere were generated by the Delaunay triangulation of randomly generated points on the surface of a unit sphere. Note that given a sufficiently large number of vertices, this procedure consistently creates valid triangulations of $2$-spheres. Likewise, triangulations of the $3$-ball were generated by the Delaunay triangulation of randomly generated points inside a unit sphere. Note also that the number of simplices scales as O($s_0$) and the small link size scales as O($1$) for the Delaunay triangulation. Thus, we expect to see linear scaling with respect to both $s$ and $s_0$ for both the $2$-sphere and $3$-ball.

\begin{figure}
\begin{center} 
\begin{subfigure}{.45\textwidth}
  \centering
  \includegraphics[width=.8\linewidth]{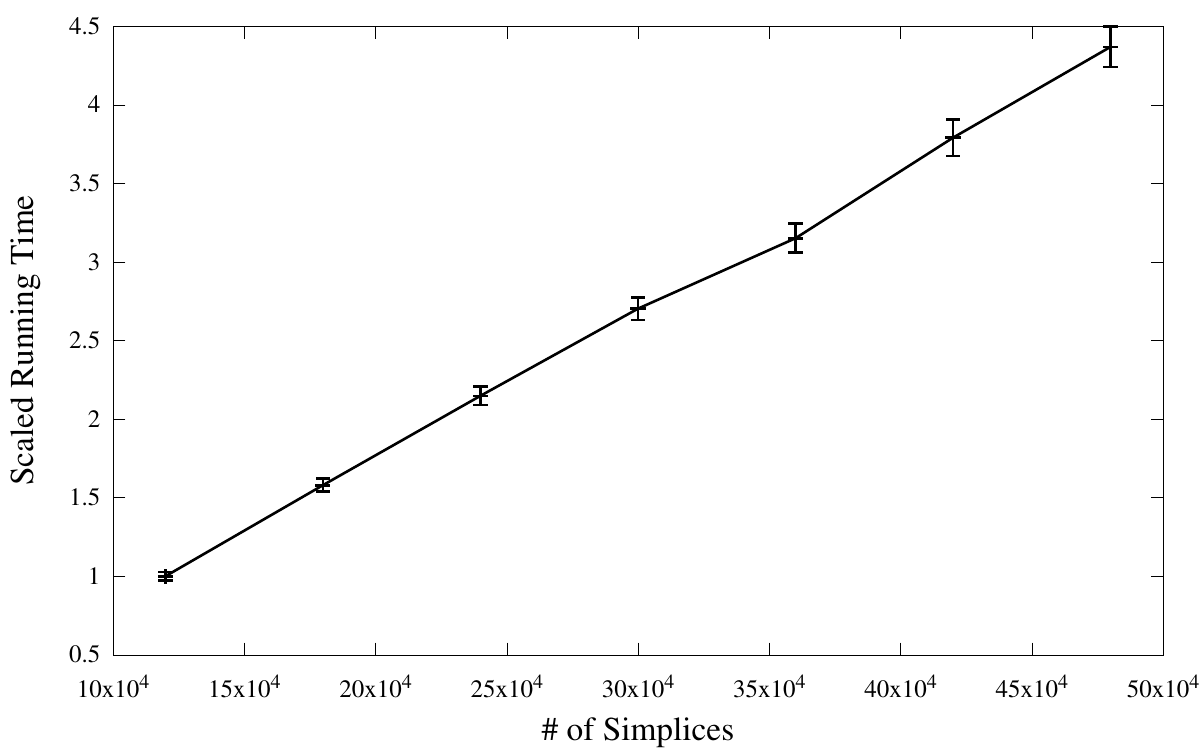}
  \caption{}\label{fig:sfig1}
\end{subfigure}%
\begin{subfigure}{.45\textwidth}
  \centering
  \includegraphics[width=.8\linewidth]{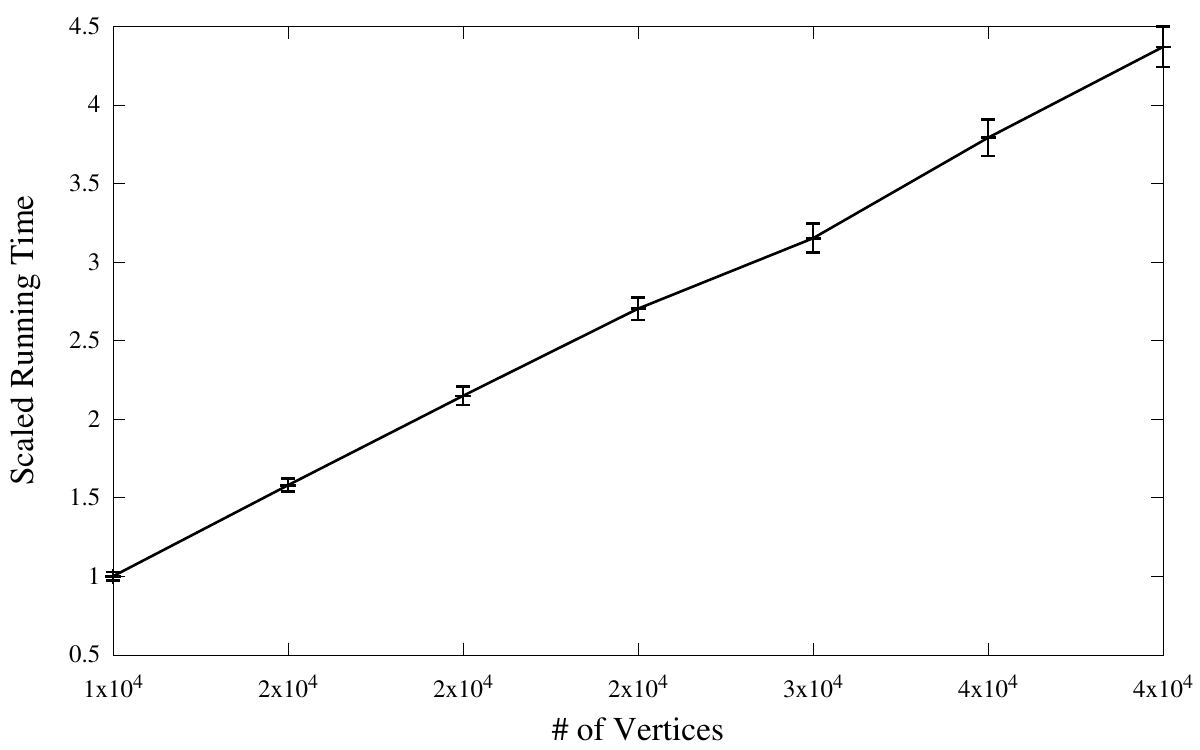}
  \caption{}\label{fig:sfig2}
\end{subfigure}
\begin{subfigure}{.45\textwidth}
  \centering
  \includegraphics[width=.8\linewidth]{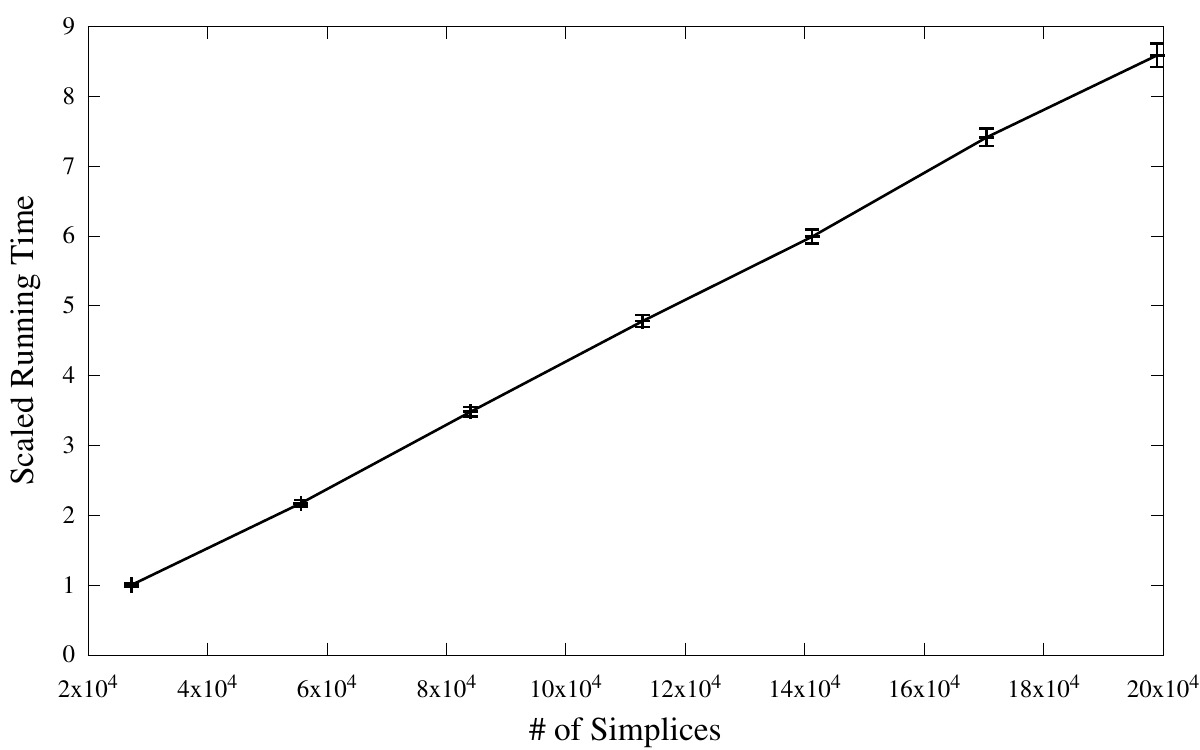}
  \caption{}\label{fig:sfig3}
\end{subfigure}
\begin{subfigure}{.45\textwidth}
  \centering
  \includegraphics[width=.8\linewidth]{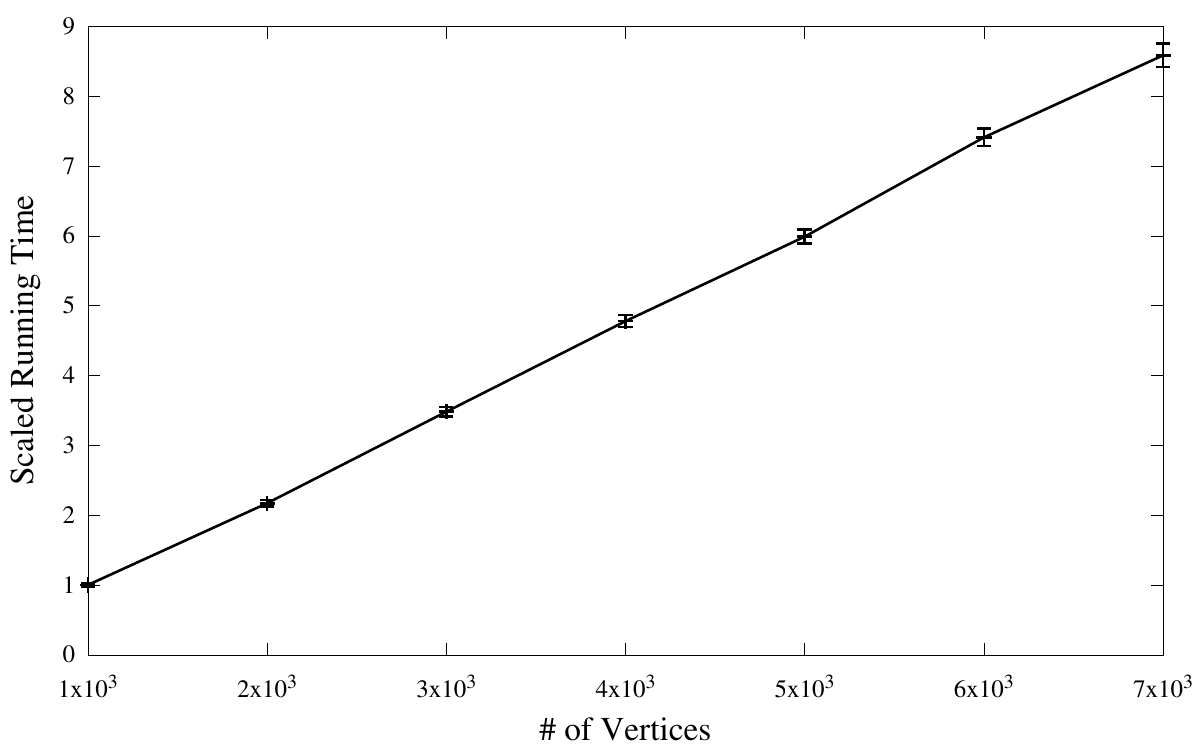}
  \caption{}\label{fig:sfig4}
\end{subfigure}
\caption{\label{fig:experiment} (A) $2$-sphere by $s$. (B) $2$-sphere by $s_0$. (C) $3$-ball by $s$. (D) $3$-ball by $s_0$. The reported values are the average time and standard deviation from 10 trials, scaled so that the first point is at 1.}
\end{center}
\end{figure}


\clearpage

\bibliographystyle{amsalpha}
\bibliography{Gcats}

\end{document}